\documentclass[11pt]{article}
\usepackage{authblk}

\usepackage[utf8]{inputenc}
\usepackage[T1]{fontenc}
\usepackage[utf8]{inputenc}
\usepackage{graphicx}
\usepackage{amsmath}
\usepackage{mathtools}
\usepackage{marginnote}
\usepackage[numbers]{natbib}

\usepackage[normalem]{ulem}

\usepackage{amsmath,amssymb,amsthm}
\usepackage{cases}
\usepackage{paralist}
\usepackage{verbatim}
\usepackage[right]{eurosym} 
\usepackage{url}
\usepackage{enumerate}

\usepackage{graphicx}

\usepackage{mathtools}
\usepackage{marginnote}
\usepackage{placeins}
\usepackage{subcaption}

\usepackage{todonotes}

\usepackage[top=2.5cm,bottom=2.5cm,left=2.5cm,right=2.5cm]{geometry}
\usepackage[colorlinks=true, allcolors=blue]{hyperref}

\newtheorem{theorem}{Theorem}[section]
\newtheorem{definition}[theorem]{Definition}
\newtheorem{proposition}[theorem]{Proposition}
\newtheorem{lemma}[theorem]{Lemma}

\newtheorem{example}[theorem]{Example}
\newtheorem{counterexample}[theorem]{Counterexample}
\newtheorem*{counterexample*}{Counterexample}
\newtheorem{conjecture}[theorem]{Conjecture} 
\newtheorem{assumption}[theorem]{Assumption}
\newtheorem{notation}[theorem]{Notation}
\newtheorem*{remark}{Remark}

\newcommand{\R}[0]{\mathbb R}
\newcommand{\N}[0]{\mathbb N}

\newcommand{\calC}[0]{\mathcal C}
\newcommand{\calH}[0]{\mathcal H}
\newcommand{\calP}[0]{\mathcal P}
\newcommand{\calV}[0]{\mathcal V}
\newcommand{\dd}[0]{\mathrm d}
\newcommand{\sumj}[0]{\sum_{j=1}^N}
\newcommand{\sumi}[0]{\sum_{i=1}^N}
\newcommand{\V}[0]{\mathcal V}
\newcommand{\fN}[0]{f^N}
\DeclareMathOperator{\dist}{dist}

\newcommand{\masssum}[2]{ \frac{1}{#2} \sum_{#1 = 1}^{#2}  }
 
\newcommand{\dx}[1]{\frac{\mathrm{d}}{\mathrm{d}#1}}

\newcommand{\norm}[1]{\left\| #1 \right\|}
\newcommand{\abs}[1]{\left| #1 \right|}

\newcommand{\Ppo}[1]{
    \ifx #1\infty
        \mathcal{P}_{c,0}
    \else 
        \mathcal{P}_{#1,0}
    \fi 
}
\newcommand{\Pp}[1]{
    \ifx #1\infty
        \mathcal{P}_{c}
    \else 
        \mathcal{P}_{#1}
    \fi 
}
\newcommand{\Wpdist}{\mathcal{W}}
\newcommand{\supp}{\mathrm{supp}}
\newcommand{\msolinit}[1]{f_0(#1)}
\newcommand{\msolt}[2]{f_{#1}(#2)}
\newcommand{\charX}[2]{R_{#1}(#2) }
\newcommand{\charV}[2]{V_{#1}(#2)}
\newcommand{\cappsi}{\Psi}
\newcommand{\Pfunc}{P}

\newcounter{partcounter}
\usepackage{bibunits} 

\title{Erratum to: Port-Hamiltonian structure \\of interacting particle systems and its mean-field limit}
\author{Jannik Daun\footnote{\href{mailto:jannik.daun@uni-wuppertal.de}{jannik.daun@uni-wuppertal.de}} , Daniel Jannik Happ\footnote{\href{mailto:dhapp@uni-wuppertal.de}{dhapp@uni-wuppertal.de}} , Birgit Jacob\footnote{ \href{mailto:bjacob@uni-wuppertal.de}{bjacob@uni-wuppertal.de}} ,
Claudia Totzeck\footnote{\href{mailto:totzeck@uni-wuppertal.de}{totzeck@uni-wuppertal.de} }}  
\affil{Port-Hamiltonian Institute, School of Mathematics and Natural Sciences, \\ University of Wuppertal, Germany}
\date{February 2026}

\begin{document}
    \maketitle
    \begin{abstract}
This erratum corrects an error in our paper on the port-Hamiltonian structure of interacting particle systems. While convergence of the gradient of the Hamiltonian remains valid under the original assumptions, relative compactness of the system trajectories in the Wasserstein space $\Pp{2}$ does not hold without an additional attractivity assumption on the binary interaction force. We provide a proof for the convergence of the gradient of the Hamiltonian based on {B}arb\u alat's lemma. A counterexample is given for the relative compactness of the system trajectories for repulsive binary interactions. In the case of short-range repulsion and long-range attraction we show several numerical studies that underpin our conjecture of relatively compact trajectories.
\end{abstract}

\begin{minipage}{0.9\linewidth}
 \footnotesize
\textbf{AMS classification:} 37K45, 82C22, 93A16.
\medskip

\noindent
\textbf{Keywords:} Port-Hamiltonian systems, interacting particle systems, mean-field limit, long-time behaviour
\end{minipage}
\phantomsection
    \begin{bibunit}[unsrtnat]
        \section{Introduction}
The main aim of this article is to correct two errors in the paper \cite{JaTo2024}. Moreover, we provide new related analytical results, a counter example, numerical simulation results and a conjecture that summarize our recent insights regarding the challenges in the analysis of interacting particle systems with bounded binary interaction potentials. In this introduction, we proceed directly by describing the main error and the corrected results, while the class of interacting particle systems and the notation is revisited in Section \ref{sec:notation}. The main error occurs in the proof of the asymptotic stability theorem both on the particle level \cite[Theorem $3.8$]{JaTo2024} and the mean-field level \cite[Theorem $4.3$]{JaTo2024}. In their original formulation, these theorems read as follows:
\begin{theorem}\label{thm:stabilor}(particle level, original formulation \cite{JaTo2024})\\
Let $\mathcal{V} \in C^{1}(\R^d; \R)$ with $\nabla \mathcal{V}$ antisymmetric, locally Lipschitz continuous and bounded.
Let $\psi \in C(\R_{\geq 0}; \R_{\geq 0})$ be bounded and locally Lipschitz continuous and assume that $\psi(x) > 0$ for all $x \in \R_{\geq 0}$. Consider the ODE system 
\begin{subnumcases}{ \label{eq:ODEintro} }
        \dx{t} r_i = v_i - \overline{v}, \text{ where } \bar{v} \coloneq \masssum{j}{N} v_j, \\
        \dx{t} v_i = \masssum{j}{N} \psi(\abs{r_j - r_i})(v_j - v_i) - \masssum{j}{N} \nabla \mathcal{V} (r_i - r_j), \\
        r_i(0) =  r_{0,i}, \quad v_i(0) =  v_{0,i}, \quad i = 1, \ldots, N \, .
\end{subnumcases}
Then for every initial condition $ z_0=(r_0,v_0)\in \R^{Nd}\times \R^{Nd}$ such that
\begin{align*}
    \masssum{j}{N} r_{0,j} = 0 \, ,
\end{align*}
the corresponding solution $z(t) = (r(t),v(t))$ of the ODE system \eqref{eq:ODEintro} satisfies
\[ \lim_{t\rightarrow \infty} \dist(z(t), L) =0, \]
where 
\[ L:= \Big\{ (\tilde r,\tilde v)\in \R^{Nd}\times \R^{Nd}\mid \tilde v_{i}=\frac{1}{N}\sumj v_j(0),\,\, \sumj \nabla \mathcal V(\tilde r_ i - \tilde r_j)=0, \, i=1,\ldots, N\Big\}. \]
\end{theorem}

\begin{theorem}\label{thm:stabilityMFor}(mean-field level, original formulation \cite{JaTo2024})\\
Let $\mathcal{V} \in C^{1}(\R^d; \R)$ with $\nabla \mathcal{V}$ antisymmetric, locally Lipschitz continuous and bounded.
Let $\psi \in C(\R_{\geq 0}; \R_{\geq 0})$ be bounded and locally Lipschitz continuous and assume that there exists $\underline{\psi} > 0$ satisfying $\psi(x) \geq \underline{\psi} > 0$ for all $x \in \R_{\geq 0}$. By $\Pp{\infty}(\R^d \times \R^d)$ and $\Pp{2}(\R^d \times \R^d)$, we denote the Wasserstein spaces of order $p = \infty$ and $p=2$, respectively. 
For every initial condition $f_0 \in \Pp{\infty} (\R^d \times \R^d)$ such that 
\begin{align*}
    \int r \, \mathrm{d} \msolinit{r,v} = 0 \, ,
\end{align*}
let $f \in C([0,\infty); \Pp{\infty}(\R^d \times \R^d) )$ be the solution to 
\begin{align}
    &\int_0^\infty \int \big( \partial_t h_t(r,v)  + (v-\bar{v}) \cdot\nabla_r h_t(r,v) \big)\mathrm{d} \msolt{t}{r,v} \mathrm{d} t + \int h_0(r,v)\mathrm{d} \msolinit{r,v} \nonumber \\
	&= \int_0^\infty \int \int\psi(| r -  r^\prime |)(v - v^\prime) - \nabla \V(r^\prime - r) \, \mathrm{d} \msolt{t}{r^\prime,v^\prime} \cdot \nabla_v h_t(r,v) \, \mathrm{d} \msolt{t}{r,v} \mathrm{d} t \, , \nonumber \\
    &\text{ for all } h \in C_c^\infty([0,\infty) \times \R^d \times \R^d), \text{ where } \bar{v} \coloneq \int v \, \mathrm{d}\msolinit{r,v} \, . \label{eq:weakformintro}
\end{align}
Then it holds that
\begin{align*}
\lim\limits_{t\rightarrow \infty}\mathrm{dist}_{\Pp{2}(\R^d \times \R^d) } (f_t, \mathcal L) =0,
\end{align*}
where 
\begin{align*}
\mathcal L = \Big\{ f \in \calP_2(\R^d \times \R^d) \mid f(x,v)=g(x)\delta(v-\bar v),\,  \int \nabla \mathcal{V}(x-\bar x)   \, \mathrm{d} g(x) = 0, \, \bar v = \int v \,  \mathrm{d} f_0(x,v) \Big\}.
\end{align*}
\end{theorem}

The main step in the proof of Theorem \ref{thm:stabilityMFor} is to derive an integral inequality for the $2$-Wasserstein distance $\Wpdist_2$, namely 
\begin{align}
     \Wpdist_{2} (f_t, \delta_{\underline{v}} )^2 \leq \Wpdist_{2} (f_0, \delta_{\underline{v}} )^2 - (2\underline{\psi} - \varepsilon) \int_0^t \Wpdist_2(f_s, \delta_{\underline{v}} )^2 \, \mathrm{d}s + \frac{\norm{\nabla \mathcal{V}}_\infty^2 t }{\varepsilon} \, . \label{eq:diffineqmf}
\end{align}
On the particle level, the corresponding inequality reads
\begin{align}
    \norm{v(t) - \textbf{1} \bar{v} }^2  \leq \norm{v(0) - \textbf{1} \bar{v} }^2 + \frac{t}{\varepsilon} \norm{\nabla \mathcal{V}}_\infty^2 - (2 \lambda_2 - \varepsilon) \int_0^t \norm{v - \textbf{1} \bar{v}}^2 \label{eq:diffineqparticle} \, .
\end{align}
The first error is that \eqref{eq:diffineqparticle} can only be derived by additionally assuming that $\psi(x) \geq \underline{\psi} > 0$ for all $ x \in \R_{\geq 0}$ rather than merely $\psi(x) > 0$, as stated in Theorem \ref{thm:stabilor}. The reason is that the second smallest eigenvalue $\lambda_2$ of $\Psi(z)$ is a function of $z$. Imposing $\psi \geq \underline{\psi}$ yields the uniform lower bound $\lambda_2(z) \geq \underline{\psi}$ for all $z$.
From this point, the proofs of Theorem $3.8$ and $4.3$ proceed similarly and therefore the second mistake is contained in both. We illustrate the issue on the mean-field level. \\ 
The conclusion from \eqref{eq:diffineqmf} that 
\begin{align*}
    \Wpdist_{2} (f_t, \delta_{\underline{v}} )^2 \leq \Big( \Wpdist_{2} (f_0, \delta_{\underline{v}} )^2 + \frac{\norm{\nabla \mathcal{V}}_\infty^2 t }{\varepsilon} \Big) \mathrm{e}^{-(2 \underline{\psi} - \varepsilon) t}
\end{align*}
is justified only if $\varepsilon \geq 2 \underline{\psi}$, but to prove the boundedness of trajectories, we would need it for $\varepsilon < 2\underline{\psi}$. 
As a result, since the argument for the exponential decay of $\Wpdist_{2} (f_t, \delta_{\underline{v}} )^2$ is invalid, we cannot use it to obtain that the family $(f_t)_{t \geq 0}$ is relatively compact in $\Pp{2}(\R^d \times \R^d)$, which is required to apply LaSalle's stability theorem. However, even without recourse to LaSalle's theorem, it is possible to show convergence of the velocities and forces. More precisely, consider the Hamiltonian 
\begin{align*}
    \mathcal{H}(f) = \frac{1}{2} \int \abs{v-\overline{v}}^2 \, \mathrm{d}f(r,v) + \frac{1}{2} \int \int \mathcal{V}(r-r^\prime) \, \mathrm{d}f(r,v) \, \mathrm{d}f(r^\prime, v^\prime), \, f \in \Pp{2}(\R^d \times \R^d) \, ,
\end{align*}
whose gradient is given by 
\begin{align*}
        \nabla \mathcal{H}(f)(r,v) = \begin{pmatrix}
            \int \nabla \mathcal{V}(r-r^\prime) \, \mathrm{d} f(r^\prime, v^\prime) \\ 
            v - \bar{v}
        \end{pmatrix} \, .
\end{align*}
The following theorem establishes the convergence of $\nabla \mathcal{H}$ to zero in $L^2$.
\begin{theorem}(Convergence of $\nabla \mathcal{H}$, new formulation)\label{thm:gradHconvergence} \\ 
     Let $\mathcal{V} \in C^{1}(\R^d; \R)$ be bounded from below with $\nabla \mathcal{V}$ antisymmetric, locally Lipschitz continuous and bounded. Let $\psi \in C(\R_{\geq 0}; \R_{\geq 0})$ be bounded and locally Lipschitz continuous and assume that there exists $\underline{\psi} > 0$ satisfying $\psi(x) \geq \underline{\psi} > 0$ for all $x \in \R_{\geq 0}$. Then for any initial condition $f_0 \in \Pp{\infty}(\R^d \times \R^d)$ satisfying $\int r \, \mathrm{d} \msolinit{r,v} = 0$, the solution $(f_t)_{t \geq 0}$ to \eqref{eq:weakformintro} satisfies
    \begin{align*}
        \int \abs{v-\bar{v}}^2 \, \mathrm{d} \msolt{t}{r,v} \stackrel{t \to \infty}{\to }0 \, .
    \end{align*}
    If additionally $\nabla \mathcal{V}$ is uniformly continuous, then we have 
    \begin{align}
        \int  \Big| \int \nabla \mathcal{V}(r-r^\prime) \, \mathrm{d} \msolt{t}{r^\prime,v^\prime} \Big|^2 \, \mathrm{d} \msolt{t}{r,v} \stackrel{t \to \infty}{\to } 0 \, . \label{eq:Forceconvergence}
    \end{align}
\end{theorem}
Theorem \ref{thm:gradHconvergence} implies a corresponding result on the particle level by taking an initial condition of the form $f_0 = \masssum{j}{N} \delta_{(r_j(0), v_j(0))}$. In this case, the convergence on the  particle level reads 
\begin{align*}
    \masssum{j}{N} \abs{v_j(t) - \overline{v}}^2 \stackrel{t \to \infty}{\to} 0, \text{ and } \frac{1}{N^3}  \sum_{j=1}^N \Big|\sum_{k=1}^N \nabla \mathcal{V}(r_j(t) - r_k(t) ) \Big|^2 \stackrel{t \to \infty}{\to} 0 \, .
\end{align*}
Since mean-field level results directly imply the corresponding particle level statements by choosing an atomic initial condition, we can restrict our discussion to the mean-field level. \\
We emphasize that Theorem \ref{thm:gradHconvergence} does not imply that $\dist(f_t, \mathcal{L}) \stackrel{t \to \infty}{\to} 0$ in $\mathcal{P}_2(\R^d \times \R^d )$, where $\mathcal{L}$ denotes the set of critical points of the Hamiltonian $\mathcal{H}$. The reason is that the family $(f_t)_{t \geq 0}$ might not be precompact in $\Pp{2}(\R^d \times \R^d)$. In fact, the convergence of the binary forces 
\begin{align*}
    \int  \Big| \int \nabla \mathcal{V}(r-r^\prime) \, \mathrm{d} \msolt{t}{r^\prime,v^\prime} \Big|^2 \, \mathrm{d} \msolt{t}{r,v} \stackrel{t \to \infty}{\to } 0
\end{align*}
can also be caused by an escape of mass to infinity, while the measure family $(f_t)_{t \geq 0}$ does not have any limit point in $\mathcal{P}_2(\R^d \times \R^d)$.  The following counterexample demonstrates that this absence of limit points can occur without additional assumptions on the binary interaction potential $\mathcal{V}$.
\begin{counterexample}(new formulation)\label{Cornoncompactness} \\
        Let $\mathcal{V} \in C^{1}(\R^d; \R)$ be bounded from below with $\nabla \mathcal{V}$ antisymmetric, uniformly continuous, locally Lipschitz continuous and bounded. Moreover, assume that $\nabla \mathcal{V}(x) \cdot x < 0$ for all $x \in \R^d\setminus\{ 0 \}$. Let $\psi \in C(\R_{\geq 0}; \R_{\geq 0})$ be bounded and locally Lipschitz continuous and assume that there exists $\underline{\psi} > 0$ satisfying $\psi(x) \geq \underline{\psi} > 0$ for all $x \in \R_{\geq 0}$.  Then there exists initial data $f_0 \in \Pp{\infty} (\R^d \times \R^d)$ such that the corresponding solution $(f_t)_{t \geq 0}$ to \eqref{eq:weakformintro} does not have a limit point in $\mathcal{P}_2(\R^d \times \R^d)$ as $t \to \infty$. That is, for any sequence $(t_n)_{n \in \N} \subseteq \R_{\geq 0} $ such that $t_n \stackrel{n \to \infty}{\to} \infty$, the sequence $(f_{t_n})_{n \in \N}$ does not converge in $\Pp{2}(\R^d \times \R^d)$.
\end{counterexample}
The condition $\nabla \mathcal{V}(x) \cdot x < 0$ in the previous counterexample indicates that the binary interaction force between two particles is repulsive for all relative positions $x$. Despite the fact that the alignment function $\psi$ is bounded from below by a positive constant, the repulsion between particles causes an escape of mass to infinity. As a consequence, relative compactness can only be obtained by an additional assumption on $\mathcal{V}$ which ensures that the interaction is sufficiently attractive. The simplest possibility is to add coercivity in the sense that $\mathcal{V}(x) \geq c \abs{x}^2 - C$ for some constants $c,C > 0$ independent of $x$. In this case, the energy balance $\dx{t} \mathcal{H}(f_t) \leq 0$ implies that the second moments of $(f_t)_{t \geq 0}$ are uniformly bounded and therefore the family is relatively compact in $\Pp{p}$ for every $p \in [1,2)$ (see Proposition \ref{prop:compactness} below). However, this approach is incompatible with our standing assumption that $\nabla \mathcal{V}$ is bounded, which allows for at most linear growth at infinity. Therefore, we adopt another approach and formulate a different attractivity condition as in the following theorem:

\begin{theorem}(new formulation)\label{thm:asymptotics} \\
        Let $\mathcal{V} \in C^{1}(\R^d; \R)$ be bounded from below by a constant $\underline{\mathcal{V}} \in \R$ and let $\nabla \mathcal{V}$ be antisymmetric, uniformly continuous, locally Lipschitz continuous and bounded. Let $\psi \in C(\R_{\geq 0}; \R_{\geq 0})$ be bounded and locally Lipschitz continuous and assume that there exists $\underline{\psi} > 0$ satisfying $\psi(x) \geq \underline{\psi} > 0$ for all $x \in \R_{\geq 0}$. Finally, let $f_0 \in \Pp{\infty}(\R^d \times \R^d)$ be an initial condition satisfying $\int r \, \mathrm{d} f_0(r,v) = 0$ and assume that there exists $r_0 \geq 0$ (possibly depending on $f_0$) such that the inequality 
        \begin{align}
            \int \int \nabla \mathcal{V}(R(r,v)-R(r^\prime,v^\prime)) \cdot (R(r,v)-R(r^\prime,v^\prime)) \, \mathrm{d} \msolinit{r,v} \, \mathrm{d}\msolinit{r^\prime, v^\prime} \geq 0 \label{eq:attrcase} 
        \end{align}
        holds for all $R \in L^2(\R^d \times \R^d, f_0; \R^d)$ satisfying 
        \begin{align*}
            \int R(r,v) \, \mathrm{d}\msolinit{r,v} = 0 \text{ and } \norm{R}_{L^2} \geq r_0 \, .
        \end{align*}
        Then the corresponding solution $(f_t)_{t \geq 0}$ to \eqref{eq:weakformintro} satisfies 
        \begin{align}
            \int \abs{r}^2 \, \mathrm{d}\msolt{t}{r,v}  &\leq \frac{1}{\underline{\psi}} \Big( 1 + 2 \sqrt{2} \big( \mathcal{H}(f_{0}) - \underline{\mathcal{V}} \big)^{\frac{1}{2}} \Big)  \max \Big\{ r_0, \Big( \int  \abs{r}^2 \, \mathrm{d}\msolinit{r,v} \Big)^{\frac{1}{2} } \Big\}  \nonumber \\
            &\quad + \frac{\norm{\psi}_\infty}{\underline{\psi}} \max \Big\{ r_0^{2}, \int \abs{r}^2 \, \mathrm{d}\msolinit{r,v} \Big\}  + \frac{2}{\underline{\psi}^{2}} (\mathcal{H}(f_0) - \underline{\mathcal{V}})   \, . \label{eq:momboundattr}
        \end{align}
        Moreover, for every $p \in [1,2)$, it holds that
        \begin{align}
             \lim_{t\rightarrow \infty} &\dist_{\Pp{p} (\R^d \times \R^d )}(f_t, \mathcal{L}) =0, \label{eq:Wpdistconv}
        \end{align}
        where
        \begin{align*}
            \mathcal L &= \Big\{ f \in \Pp{2}(\R^d \times \R^d) \mid f=g(r) \delta(v-\bar{v} ), \int \nabla \mathcal{V}(r -r^\prime) \mathrm{d} g(r^\prime) = 0,  \bar{v} = \int  v \,  \mathrm{d} \msolinit{r,v} \Big\}. 
        \end{align*}
    \end{theorem}
    For a radially symmetric potential $\mathcal{V}(x) = U (\abs{x})$ with $U \in C^1(\R_{\geq 0}, \R), \, U^\prime(0) = 0$, the condition \eqref{eq:attrcase} is clearly satisfied if $U^\prime(x) \geq 0 , \, x \in \R_{\geq 0}$, meaning that the forces are strictly attractive. However, we emphasize that the long-range attractivity condition $U^\prime(x) \geq 0$ for all $x \geq r_0$ and some $r_0 > 0$ does in general not imply \eqref{eq:attrcase} (see Example \ref{ex:morsecounterex} below). However, we conjecture that boundedness of second moments is still satisfied in this case. 
\begin{conjecture}\label{con:boundedness}
   Let $\psi \in C(\R_{\geq 0}; \R_{\geq 0})$ be bounded, locally Lipschitz continuous and assume that there exists $\underline{\psi} > 0$ satisfying $\psi(x) \geq \underline{\psi} > 0$ for all $x \in \R_{\geq 0}$. Let $\mathcal{V}(x)  = U(\abs{x}), x \in \R^d $ for a function $U \in C^1(\R_{\geq 0}, \R)$ with $U^\prime (0) = 0$, $U^\prime$ bounded and such that $\nabla \mathcal{V}$ is locally Lipschitz continuous. Moreover, let there exist $r_0 \geq 0$ such that $U^\prime(x) \geq 0$ for all $x \geq r_0$. Then for every initial condition $f_0 \in \Pp{\infty}(\R^d \times \R^d)$ with $\int r \, \mathrm{d} \msolinit{r,v} = 0$, the corresponding solution $(f_t)_{t \geq 0}$ to \eqref{eq:weakformintro} satisfies 
    \begin{align*}
        \sup_{t \geq 0} \int \abs{r}^2 \, \mathrm{d}\msolt{t}{r,v} < \infty \, .
    \end{align*}
\end{conjecture}
The remainder of the article is structured as follows: The relevant background on interacting particle systems and the notation are revisited in Section \ref{sec:notation}, while Section \ref{sec:newresults} is devoted to prove the new results, namely Theorem \ref{thm:gradHconvergence}, Counterexample \ref{Cornoncompactness} and Theorem \ref{thm:asymptotics}. Numerical simulation results to assess Conjecture \ref{con:boundedness} are shown in Section \ref{sec:numerics}.

\section{Description of the model and notation} \label{sec:notation}
    We will use the same notation as in the previous paper \cite{JaTo2024}. For the sake of completeness, we recall the formulation of the interacting particle system. We consider $N \in\N, \,  N \geq 2$ interacting particles moving in $\R^d$. The state of each particle is given by a position $x_i$ and a velocity coordinate $v_i$ for $i = 1, \ldots, N$. Those states evolve according to the Newtonian dynamics
    \begin{subnumcases}{ \label{eq:xvODE} }
        \dx{t} x_i = v_i, \\
        \dx{t} v_i = \masssum{j}{N} \psi(\abs{x_j - x_i})(v_j - v_i) - \masssum{j}{N} \nabla \mathcal{V} (x_i - x_j), \\
        x_i(0) =  \hat{x}_i, \quad v_i(0) =  \hat{v}_i, \quad i = 1, \ldots, N \, .
    \end{subnumcases}
    Here, $\psi \colon \R \rightarrow \R_{\ge 0}$ models the velocity alignment strength and $\mathcal{V} \colon \R^d  \rightarrow \R$ denotes the binary interaction potential among the particles.  
    We assume that $\mathcal{V} \in C^{1}(\R^d ; \R) $ and  
    \begin{align}
        \nabla \V(x) = - \nabla \V(-x) \text{ for all } x\in \R^d. \label{eq:antisymmetry}
    \end{align}
    The antisymmetry condition \eqref{eq:antisymmetry} ensures that the binary interaction forces between two particles have the same magnitude but opposite directions, in accordance with Newton's third law. As a consequence, the mean velocity 
    \begin{align*}
        \bar{v} \coloneq \masssum{j}{N} v_j(t)
    \end{align*}
    is a conserved quantity and the center of mass 
    \begin{align*}
        \bar{x}(t) \coloneq \masssum{j}{N} x_j(t)
    \end{align*}
     satisfies $\bar{x}(t) = \bar{x}(0) + t \bar{v} $. Therefore, we introduce the position coordinates relative to the center of mass by $r_i \coloneq x_i - \bar{x}, \, i = 1, \ldots, N$. \\
    The \textit{mean-field limit} is obtained by taking the limit $N \to \infty$ of the empirical measure
    \begin{align}
        f_t(r,v) \coloneq \masssum{i}{N}  \delta (r - r_i(t)) \otimes \delta (v - v_i(t)) \, . \label{eq:empmeasure}
    \end{align}
    To rigorously describe limits of probability measures, we use the \textit{Wasserstein} distance. For $p \in [1,\infty)$, let $\Pp{p}(\R^d \times \R^d)$ be the set of all Borel probability measures with finite $p$-th moment. The $p = \infty$ Wasserstein space is denoted by $\Pp{\infty}$ and contains all compactly supported probability measures. For $p \in [1,\infty]$ and $\mu,\nu\in \Pp{p}(\R^d \times \R^d )$ (respectively $\Pp{\infty} (\R^d \times \R^d$)), the corresponding Wasserstein distance is defined as
    \begin{align*}
        \Wpdist_p (\mu, \nu) := 
        \begin{cases}
            \inf\limits_{\pi \in \Pi(\mu, \nu)}  \Big( \displaystyle \int \abs{(x - x^\prime, y - y^\prime) }^p \, \mathrm{d} \pi(x,y,x^\prime, y^\prime) \Big)^{\frac{1}{p} }   , &1 \leq p < \infty, \\
            \inf\limits_{\pi \in \Pi(\mu, \nu)} \sup \,  \{ \abs{(x-x^\prime, y-y^\prime)} \mid (x,y,x^\prime,y^\prime) \in \supp(\pi) \} \, , &p = \infty \, ,
        \end{cases}
    \end{align*}
    where $\Pi(\mu,\nu)$ denotes the set of all Borel probability measures on $ \R^{4d}$ that have $\mu$ and $\nu$ as first and second marginals respectively, i.e.\    
    \begin{align*}
        \pi(B \times \R^d \times \R^d ) = \mu(B), \qquad \pi(\R^d \times \R^d \times B) = \nu(B) \quad \text{for } B \in \mathcal{B}(\R^{d} \times \R^d ) \, .
     \end{align*}
    By $\Ppo{p}$ (resp.\ $\Ppo{\infty}$), we denote the set of all $\mu \in \Pp{p}$ (resp.\ $\Pp{\infty}$) which satisfy 
    \begin{align*}
        \int r \, \mathrm{d} \mu(r,v) = 0 \, .
    \end{align*}
    Since the empirical measure \eqref{eq:empmeasure} is defined in terms of the center of mass coordinates $r_i$, it follows that $(f_t)_{t \geq 0} \subseteq \Ppo{\infty}$. \\
    To obtain an equation for the mean-field limit, we note that the empirical measure \eqref{eq:empmeasure} satisfies
    \begin{align}
        &\int_0^T \int \big( \partial_t h_t(r,v)  + (v-\bar{v}) \cdot\nabla_r h_t(r,v) \big)\mathrm{d} \msolt{t}{r,v} \mathrm{d} t + \int h_0(r,v)\mathrm{d} \msolinit{r,v} \nonumber \\
	&= \int_0^T \int \int\psi(| r -  r^\prime |)(v - v^\prime) - \nabla \V(r^\prime - r) \, \mathrm{d} \msolt{t}{r^\prime,v^\prime} \cdot \nabla_v h_t(r,v) \, \mathrm{d} \msolt{t}{r,v} \mathrm{d} t \label{eq:weakform}
    \end{align}
    for every test function $h \in C_c^\infty (\R_{\geq 0} \times \R^d \times \R^d )$, where $\bar{v} = \int v^\prime \, \mathrm{d}\msolinit{r^\prime, v^\prime}$. This weak formulation \eqref{eq:weakform} can be used to define a measure-valued solution corresponding to an arbitrary initial condition $f_0 \in \Ppo{\infty}(\R^d \times \R^d)$. Namely, we require that $f \in C( [0,T) ; \Ppo{\infty} (\R^d \times \R^d) )$ satisfies \eqref{eq:weakform} for every test function $h$. \\
    The characteristics $(R,V) \in C^1([0,T) ; C(\supp(f_0) ; \R^d \times \R^d ) )$ corresponding to an initial condition $f_0 \in \Ppo{\infty}(\R^d \times \R^d)$ are defined as solutions to the Banach space-valued ODE system
    \begin{subnumcases}{\label{eq:ODEchar} \qquad }
        \dx{t} \charX{t}{r,v} = \charV{t}{r,v} - \bar{v} \, , \label{EqODEX} \\ 
        \dx{t} \charV{t}{r,v} = \int \psi (\abs{ \charX{t}{r,v} - \charX{t}{r^\prime, v^\prime } } ) (\charV{t}{r^\prime, v^\prime} - \charV{t}{r,v} ) \, \mathrm{d}\msolinit{r^\prime, v^\prime} \nonumber \\
        \qquad \qquad \qquad \quad -  \int  \nabla \mathcal{V} ( \charX{t}{r,v} - \charX{t}{r^\prime, v^\prime} ) \, \mathrm{d} \msolinit{r^\prime, v^\prime} \, , \label{EqODEV} \\ 
        \charX{0}{r,v} = r, \, \charV{0}{r,v} = v , \, (r,v) \in \supp (f_0) \, .
    \end{subnumcases} 
    Under appropriate assumptions (the standing Assumptions \ref{ass:asymptotics} we use in Section \ref{sec:newresults} are sufficient), the system \eqref{eq:ODEchar} admits a unique global solution. In addition, the weak formulation \eqref{eq:weakform} also admits a unique measure-valued solution $f$, which is given by
    \begin{align}
        f_t = (\charX{t}{\cdot,\cdot}, \charV{t}{\cdot, \cdot} )_{\#} f_0 \label{eq:pushforward} \, , \, t \geq 0 \, .
    \end{align}
    We conclude this section by recalling that the \textit{Hamiltonian} $\mathcal{H}$ is defined as 
    \begin{align*}
        \mathcal{H}(f) \coloneq \frac{1}{2} \int \abs{v-\bar{v}} \, \mathrm{d} f(r,v) + \frac{1}{2} \int \int \mathcal{V}(r-r^\prime) \, \mathrm{d} f(r,v) \, \mathrm{d}f(r^\prime, v^\prime) \, .
    \end{align*}
    The gradient of the Hamiltonian is given by (see \cite{JaTo2024})
    \begin{align*}
        \nabla \mathcal{H}(f)(r,v) = \begin{pmatrix}
            \int \nabla \mathcal{V}(r-r^\prime) \, \mathrm{d} f(r^\prime, v^\prime) \\ 
            v - \bar{v}
        \end{pmatrix} \, .
    \end{align*}
    The following proposition establishes the energy balance and conservation of the mean velocity on the mean-field level: 
    \begin{proposition}(energy balance, see \cite[Theorem $4.1$]{JaTo2024})\label{prop:Hbalance} \\
        For every initial condition $f_0 \in \Ppo{\infty}(\R^d \times \R^d)$ and all $t \geq 0$, we have 
        \begin{align}
            \dx{t} \mathcal{H}(f_t) = -\frac{1}{2} \int \int \psi (\abs{r-r^\prime}) \abs{v-v^\prime}^2 \, \mathrm{d}\msolt{t}{r,v} \, \mathrm{d} \msolt{t}{r^\prime, v^\prime}  \label{eq:enbalance} \, .
        \end{align}
        Moreover, the mean velocity is a conserved quantity, that is 
        \begin{align*}
            \int v \, \mathrm{d} \msolt{t}{r,v} = \int v \, \mathrm{d} \msolinit{r,v} \eqcolon \overline{v} \text{ for all } t \geq 0 \, .
        \end{align*}
    \end{proposition}

\newpage

\section{New results}\label{sec:newresults}
Throughout this chapter, we work under the following standing assumptions.
\begin{assumption}\label{ass:asymptotics} 
   \begin{enumerate}
       \item The potential $\mathcal{V} \in C^1(\R^d ; \R)$ is bounded from below by a constant $\underline{\mathcal{V}} \in \R$. Moreover, we assume that $\nabla \mathcal{V}$ is antisymmetric, locally Lipschitz continuous and bounded.
       \item The alignment function $\psi \in C(\R_{\geq 0};\R_{\geq 0})$ is bounded, locally Lipschitz continuous and there exists $\underline{\psi} > 0$ such that $\psi(x) \geq \underline{\psi}$ for all $x \in \R_{\geq 0}$.
   \end{enumerate}
\end{assumption}
The following proposition is a direct consequence of Danskin's theorem, see \cite[Theorem 1]{Danskin}.
\begin{proposition}\label{prop:Danskin}
    Let $n , m \in \N$, $K \subseteq \R^m $ compact and let $\varPhi \colon \R^n\to \R$ be a $C^{1}$ function. If $I \subseteq \R$ is a nonempty open interval and $F \in C^{1}(I; C(K; \R^n) )$, then the map $S \colon I \to \R, \, S(t) = \sup_{x \in K} \varPhi(F_t(x))$ is right differentiable at every $t \in I$ and its right-derivative $\dx{t}_+ S(t)$ is given by 
    \begin{align*}
        \dx{t}_+ S(t) \coloneq \lim\limits_{h \downarrow 0} \frac{S(t + h) - S(t)}{h} =  \sup\limits_{\substack{x \in K \\ \varPhi(F_t(x)) = S(t) }} \nabla \varPhi(F_t(x)) \cdot \dx{t} F_t(x) \, . 
    \end{align*}
\end{proposition}

\begin{lemma}\label{LVbound}
	For every initial condition $f_0 \in \Ppo{\infty}(\R^d \times \R^d)$, let $(f_t)_{t \geq 0}$ and $(V_t)_{t \geq 0}$ denote the solutions to \eqref{eq:weakform} and \eqref{eq:ODEchar}, respectively. Then for all $t \geq 0$, it holds that
    \begin{align}
        \int \abs{v-\bar{v}}^2 \, \mathrm{d}\msolt{t}{r,v} &\leq 2 (\mathcal{H}(f_0) - \underline{\mathcal{V}}) \, , \label{eq:vL2bound} \\ 
		\norm{\charV{t}{\cdot, \cdot}  - \bar{v} }_\infty ^{2} &\leq  \Big(  \norm{\charV{0}{\cdot, \cdot } - \bar{v} }_\infty   \mathrm{e}^{-\underline{\psi} t} + \frac{\big\| \nabla \mathcal{V} \big\|_{\infty}}{\underline{\psi}} \big( 1 - \mathrm{e}^{-\underline{\psi} t} \big) \Big)^{2}  \label{eq:vLinftybound} \, .
	\end{align}
\end{lemma}
\begin{proof}
    The inequality \eqref{eq:vL2bound} is an immediate consequence of the fact that the function $t \mapsto \mathcal{H}(f_t)$ is decreasing (see Proposition \ref{prop:Hbalance}) and that $\mathcal{V}$ is bounded from below. \\
	In order to prove the inequality \eqref{eq:vLinftybound}, we define the function $S(t) := \norm{ \charV{t}{\cdot, \cdot} - \bar{v}}_\infty^{2}$ for  $t \geq 0$. By applying Proposition \ref{prop:Danskin} to the function $S$ with $K = \supp (f_0)$, we obtain that 
			\begin{align*}
				\dx{t}_+& S(t) = \lim\limits_{h \downarrow 0} \frac{S(t+ h) - S(t) }{h} =   
				 \sup\limits_{\substack{(r,v) \in \supp (f_0) \\
						\abs{\charV{t}{r,v}  - \bar{v} }^{2} = S(t)}  } 2 (\charV{t}{r,v} - \bar{v}) \cdot \dx{t} \charV{t}{r,v}  \\
				 &=  \sup\limits_{\substack{(r,v) \in \supp (f_0) \\
						\abs{ \charV{t}{r,v}  - \bar{v} }^{2} = S(t)}  } 
				2\Big[ \displaystyle \int  \psi \big( \abs{\charX{t}{r,v} - \charX{t}{r^\prime, v^\prime} } \big) (\charV{t}{r,v}  - \bar{v}) \cdot (\charV{t}{r^\prime, v^\prime} -  \charV{t}{r,v} ) \, \mathrm{d}\msolinit{r^\prime, v^\prime}    \\
				&\qquad \qquad\qquad  - \int \nabla \mathcal{V}( \charX{t}{r,v}  - \charX{t}{x^\prime, v^\prime } ) ) \cdot ( \charV{t}{r,v}  - \bar{v}) \, \mathrm{d}\msolinit{r^\prime, v^\prime} \Big] \\
				&\leq  \sup\limits_{\substack{(r,v) \in \supp (f_0) \\
						\abs{ \charV{t}{r,v}  - \bar{v} }^{2} = S(t)}  } - 2\underline{\psi} \abs{ \charV{t}{r,v}  - \bar{v}}^{2} - 2\int \vspace{-0.2em} \nabla \mathcal{V}( \charX{t}{r,v} - \charX{t}{r^\prime, v^\prime} ) \cdot ( \charV{t}{r,v} - \bar{v}) \, \mathrm{d}\msolinit{r^\prime, v^\prime}  \\
				&\leq - 2\underline{\psi} S(t) + 2\big\| \nabla \mathcal{V} \big\|_{\infty} \sqrt{S(t)} \, .
			\end{align*}
	In the third line above, we have used that 
    \begin{align*}
        (\charV{t}{r,v}  - \bar{v}) \cdot (\charV{t}{r^\prime, v^\prime} -  \charV{t}{r,v} ) 
        &= (\charV{t}{r,v}  - \bar{v}) \cdot (\charV{t}{r^\prime,v^\prime}  - \bar{v}) - \abs{\charV{t}{r,v}  - \bar{v}}^2 \\
        &\leq \abs{\charV{t}{r,v}  - \bar{v}} \big( \abs{ \charV{t}{r^\prime, v^\prime} - \bar{v} } - \abs{\charV{t}{r, v} - \bar{v}} \big)
        \leq 0 ,
    \end{align*}
    whenever $S(t) = \abs{\charV{t}{r, v} - \bar{v}}^2$. 
    We note that the maximal solution, that is, a solution on a maximal existence interval that pointwise dominates any other solution with the same initial condition (see \cite[Chapter 3.2]{Hartmann}), of the scalar ODE 
	\begin{align*}
		\dx{t} x(t) = - 2 \underline{\psi} x(t) + 2\big\| \nabla \mathcal{V} \big\|_{\infty} \sqrt{\abs{x(t)}}
	\end{align*}
	with initial condition $x(0) \geq 0$ is given by 
	\begin{align*}
		x(t) &= \Big( \sqrt{x(0)} \mathrm{e}^{-\underline{\psi} t} + \big\| \nabla \mathcal{V} \big\|_{\infty} \int_0^t \mathrm{e}^{-\underline{\psi} (t-s))} \, \mathrm{d}s \Big)^{2} 
		= \Big( \sqrt{x(0)} \mathrm{e}^{-\underline{\psi} t} + \frac{\big\| \nabla \mathcal{V} \big\|_{\infty}}{\underline{\psi}} (1- \mathrm{e}^{-\underline{\psi} t)} ) \Big)^{2} .
	\end{align*}
	Using the comparison principle for scalar ODEs (see e.\ g.\ \cite[Theorem $4.1$, Chapter $3$]{Hartmann}), we obtain that
	\begin{align*}
		S(t) &\leq
		\Big( \sqrt{S(0)} \mathrm{e}^{- \underline{\psi} t} + \frac{\big\| \nabla \mathcal{V} \big\|_{\infty}}{\underline{\psi}} \big( 1 - \mathrm{e}^{-\underline{\psi} t)} \big) \Big)^{2} 
		 , \,  t \geq 0 \, . \qquad \qedhere
	\end{align*}
\end{proof}
We note that the supremum norm in \eqref{eq:vLinftybound} can also be replaced by an $L^2(f_0;\R^d)$-norm using a similar argument. \\
Our main tool to prove the convergence of $\nabla \mathcal{H}$ as stated in Theorem \ref{thm:gradHconvergence} is {B}arb\u alat's lemma.  We refer to \cite[Theorem $4$]{FarkasWegner} for a proof of the lemma.

\begin{lemma}({B}arb\u alat's lemma) \label{Lbarbalat}\\  
    Let $E$ be a Banach space and $g \in C^{1}(\R_{\geq 0}; E)$ such that the limit $\lim\limits_{t \to \infty} g(t) = g_\infty$ exists in $E$. If $g^\prime \colon [0,\infty) \to E$ is uniformly continuous, then $ g^\prime (t) \to  0 $ in $E$ as $t \to \infty$.
    \end{lemma}
    
\begin{proof}[Proof of Theorem \ref{thm:gradHconvergence}]
    We start the proof by showing $L^2$-convergence of $V_t$. Due to the energy balance \eqref{eq:enbalance}, we have 
        \begin{align*}
            \dx{t} \calH (f_t) &= - \frac{1}{2} \int \int \psi \big( \abs{\charX{t}{r,v} - \charX{t}{r^\prime, v^\prime} } \big) \abs{ \charV{t}{r,v} - \charV{t}{r^\prime, v^\prime} }^{2} \mathrm{d} \msolinit{r,v} \mathrm{d} \msolinit{r^\prime, v^\prime} \\
            &\leq - \frac{\underline{\psi}}{2} \int \int \abs{ \charV{t}{r,v} - \charV{t}{r^\prime, v^\prime}  }^{2} \, \mathrm{d} \msolinit{r,v} \, \mathrm{d}\msolinit{r^\prime, v^\prime} \\
            &= - \underline{\psi} \int \abs{ \charV{t}{r,v} - \bar{v}}^{2} \, \mathrm{d}\msolinit{r,v} \eqcolon - \underline{\psi} g(t)\, . 
        \end{align*}
        Since the potential $\mathcal{V}$ is bounded from below, also the Hamiltonian $\calH$ is bounded from below. Hence, we conclude that $g \in L^{1}([0,\infty))$. Moreover, Lemma \ref{LVbound} implies that $\sup\limits_{t \geq 0} \norm{V_t -\bar{v}}_{L^2} < \infty$. The derivative of $g$ is given by 
        \begin{align*}
            g^\prime(t) = - \int \int \psi \big(\abs{ \charX{t}{r,v} - \charX{t}{r^\prime, v^\prime} }\big) \abs{ \charV{t}{r,v} - \charV{t}{r^\prime, v^\prime} }^2 \, \mathrm{d} \msolinit{r,v} \mathrm{d} \msolinit{r^\prime, v^\prime} \\
            - 2 \int \charV{t}{r,v} \cdot \int \nabla \mathcal{V} (\charX{t}{r,v} - \charX{t}{r^\prime, v^\prime}) \, \mathrm{d} \msolinit{r,v} \, \mathrm{d} \msolinit{r^\prime, v^\prime}  
        \end{align*}
        and the right hand side is bounded uniformly with respect to $t$ since $\nabla \mathcal{V}$ is bounded. Hence, $g$ is Lipschitz continuous and therefore uniformly continuous. Since $g \in L^1([0,\infty))$, {B}arb\u alat's Lemma \ref{Lbarbalat} implies that $g(t) \stackrel{t \to \infty}{\to} 0$. Since
        $f_t = (\charX{t}{\cdot,\cdot} , \charV{t}{\cdot, \cdot} )_{\#} f_0 $, we obtain that
        \begin{align*}
            \Wpdist_{2}^{2} ( (p_v)_{\#} f_t, \delta_{\bar{v}} ) = \int \abs{v - \bar{v}}^{2} \, \mathrm{d} \msolt{t}{r,v} = \int \abs{\charV{t}{r,v} - \bar{v} }^{2} \, \mathrm{d}\msolinit{r,v} \stackrel{t \to \infty}{\to } 0 \, . 
        \end{align*}
        To complete the proof, let us now additionally assume that $\nabla \mathcal{V}$ is uniformly continuous and we have to show convergence of the forces as stated in \eqref{eq:Forceconvergence}. Note that the embedding \\$C(\supp(f_0); \R^d) \hookrightarrow L^{2} (f_0;\R^d )$ is continuous and therefore \\
        $(\charV{t}{\cdot,\cdot} - \bar{v}) \in C^{1}(\R_{\geq 0}; C(\supp(f_0); \R^d) ) \subseteq C^{1}(\R_{\geq 0}; L^{2}(f_0;\R^d))$. We have already shown that $ \charV{t}{\cdot,\cdot} - \bar{v} \to 0$ in $L^{2}(f_0; \R^d)$ as $t \to \infty$. In order to apply {B}arb\u alat's lemma \ref{Lbarbalat}, we have to prove that $\dx{t} ( \charV{t}{\cdot,\cdot} - \bar{v}) \in C(\R_{\geq 0}; L^2(f_0;\R^d))$ is uniformly continuous, where
         \begin{align}
             \dx{t} ( \charV{t}{r,v} - \bar{v}) 
             &= \int \psi (\abs{ \charX{t}{r,v}-  \charX{t}{r^\prime, v^\prime} } ) ( \charV{t}{r^\prime, v^\prime} - \charV{t}{r,v} ) \, \mathrm{d}\msolinit{r^\prime, v^\prime} \nonumber \\
             &\quad - \int \nabla \mathcal{V} ( \charX{t}{r,v} - \charX{t}{r^\prime, v^\prime } ) \, \mathrm{d}\msolinit{r^\prime, v^\prime} \eqcolon A_{1,t}(r,v) + A_{2,t}(r,v) \, .
         \end{align}
         We verify separately for both terms that $A_1$ and $A_2$ are uniformly continuous. We claim that $A_1$ tends to zero in $L^{2}(f_0;\R^d)$ as $t \to \infty$. Then it follows that the continuous function $A_1 \in C(\R_{\geq 0}; L^2(f_0;\R^d))$ is even uniformly continuous. We have 
         \begin{align*}
              \norm{A_{1,t}}_{L^2}^2 = \int \Big|& \int \psi \big(\abs{R_t(r,v) - R_t(r^\prime, v^\prime)} \big) (V_t(r^\prime, v^\prime) - V_t(r,v)) \, \mathrm{d}\msolinit{r^\prime, v^\prime} \Big|^{2} \, \mathrm{d}\msolinit{ r,v} \\
              &\leq \norm{\psi}_\infty^{2} \int \int \big| \charV{t}{r^\prime, v^\prime} - \charV{t}{r,v} \big|^{2} \, \mathrm{d}\msolinit{r^\prime, v^\prime} \, \mathrm{d}\msolinit{r,v} \\
              &= 2 \norm{\psi}_\infty^{2} \int \big| \charV{t}{r,v} - \bar{v} \big|^{2} \, \mathrm{d} \msolinit{r, v} \stackrel{t \to \infty}{\to} 0 \, .
         \end{align*}
         To show uniform continuity of $A_2$, we fix $\varepsilon > 0$. Since $\nabla \mathcal{V}$ is uniformly continuous, there exists $\delta > 0$ such that $\abs{\nabla \mathcal{V}(x) - \nabla \mathcal{V}(y) } \leq \varepsilon $ whenever $x,y \in \R^d$ with $\abs{x-y} \leq \delta$. According to Lemma \ref{LVbound}, we have 
         \begin{align*} 
             L := \sup_{t \geq 0} \norm{\charV{t}{\cdot, \cdot} - \bar{v}}_\infty < \infty .
         \end{align*}
         In particular, for all pairs $(r,v), (r^\prime, v^\prime) \in \supp(f_0)$, and  $s,t \in \R_{\geq 0}$, we have 
         \begin{align*} 
             &\abs{ \charX{t}{r,v} - \charX{t}{r^\prime, v^\prime} - (\charX{s}{r,v} - \charX{s}{r^\prime, v^\prime} ) } \\
             &\quad \leq \abs{ \charX{t}{r,v} - \charX{s}{r,v} } +
             \abs{ \charX{t}{r^\prime,v^\prime} - \charX{s}{r^\prime,v^\prime} }
             \leq 2 L \abs{t-s} \, .
         \end{align*}
         Thus, for all $s,t \in \R_{\geq 0}$ with $\abs{t-s} \leq \delta /{2L}$, we obtain 
         \begin{align*}
             \norm{A_{2,t} - A_{2,s}}_{L^2}^2 
             &= \int \Big| \int \nabla \mathcal{V}( \charX{t}{r,v} - \charX{t}{r^\prime, v^\prime} ) - \nabla \mathcal{V}( \charX{s}{r,v} - \charX{s}{r^\prime, v^\prime} ) \,  \mathrm{d} \msolinit{r^\prime, v^\prime} \Big|^{2} \, \mathrm{d}\msolinit{r,v} \\
             & \leq \int  \int \big|\nabla \mathcal{V}( \charX{t}{r,v} - \charX{t}{r^\prime, v^\prime} ) - \nabla \mathcal{V}( \charX{s}{r,v} - \charX{s}{r^\prime, v^\prime} ) \big|^{2} \,  \mathrm{d}\msolinit{r^\prime, v^\prime} \, \mathrm{d}\msolinit{r,v} \\
             &\leq \varepsilon^2 \, .
         \end{align*}
         Thus, we have shown that $\dx{t}(V_t - \bar{v})$ is uniformly continuous with values in $L^{2} (f_0;\R^d)$. {B}arb\u alat's Lemma \ref{Lbarbalat} implies that $\dx{t}(V - \bar{v}) \stackrel{t \to \infty}{\to} 0$ in $L^{2} (f_0;\R^d)$. Since we have already shown that $A_{1,t}$ vanishes as $t \to \infty$, we also obtain that the force term $A_{2,t}$ converges to zero. 
\end{proof}

    Having established Theorem \ref{thm:gradHconvergence}, we now turn to the remaining proofs of Counterexample \ref{Cornoncompactness} and Theorem \ref{thm:asymptotics}. Both proofs exploit the time evolution of a suitable weighted second moment of the position coordinate, which we will introduce below.
	\begin{notation}
		We introduce the functions 
		\begin{align}
			\cappsi \colon \R_{\geq 0} &\to \R_{\geq 0}, \, x \mapsto \int_0^x r \psi(r) \, \mathrm{d}r, \\
			 \Pfunc \colon \Pp{2}(\R^d \times \R^d) &\to \R, \,
            f \mapsto \int (v-\bar{v}) \cdot r \, \mathrm{d} f(r,v) + \frac{1}{2} \int \int \cappsi \big( \abs{r-r^\prime)} \big) \, \mathrm{d}f(r,v) \, \mathrm{d}f(r^\prime, v^\prime) \, , \label{eq:defP} \\ 
            \text{ where } \bar{v} &= \int v \, \mathrm{d}f(r,v) \, . \nonumber
		\end{align} 
        Due to $0< \underline{\psi} \leq \psi(x) \leq \norm{\psi}_\infty$ for all $ x \in \R_{\geq 0}$, we obtain the corresponding bound $\underline{\psi} x^{2} \leq 2\Psi(x) \leq \norm{\psi}_\infty x^2$. Thus, all integrals in \eqref{eq:defP} are well-defined and we also conclude that $\Pfunc$ is continuous with respect to $\Pp{2}$-convergence (see \cite[Theorem $6.9$]{Villani}).  
	\end{notation}
    The functional $\Pfunc$ plays an important role in the stability analysis since the second moment of the position coordinates can be upper bounded by $\Pfunc$. Moreover, it is a straightforward computation using the characteristic equations \eqref{eq:ODEchar} to compute the derivative of $\Pfunc$ along trajectories.
	\begin{proposition}\label{prop:derivPfunctional}
		For every initial condition $f_0 \in \Ppo{\infty}(\R^d \times \R^d)$, let $(f_t)_{t \geq 0}$ denote the solution to \eqref{eq:weakform}. Then for $t \geq 0$, it holds that
		\begin{align}
			\dx{t} \Pfunc (f_t) = \int \abs{ v - \bar{v}  }^{2} \, \mathrm{d} \msolt{t}{r,v} - \frac{1}{2} \int \int (r-r^\prime) \cdot \nabla \mathcal{V}(r-r^\prime) \, \mathrm{d} \msolt{t}{r,v} \, \mathrm{d}\msolt{t}{r^\prime, v^\prime} \, . \label{eq:Pderiv}
		\end{align} 
	\end{proposition}
    We can now restate and prove Counterexample \ref{Cornoncompactness} in a slightly more explicit way. 
\begin{counterexample*}(Refined version of Counterexample \ref{Cornoncompactness}) \\
        Assume that $\nabla \mathcal{V}$ is uniformly continuous and $\nabla \mathcal{V}(x) \cdot x < 0$ for all $x \in \R^d\setminus\{ 0 \}$. Then for every $f_0 \in \Ppo{\infty} (\R^d \times \R^d)$ which satisfies either $P(f_0) > 0$ or $P(f_0) = 0$ and $\frac{\mathrm{d}}{\mathrm{d}t} \Big|_{t = 0} P(f_0) > 0$, the corresponding solution $(f_t )_{t \geq 0}$ to \eqref{eq:weakform} does not have a limit point in $\mathcal{P}_2(\R^d \times \R^d)$ as $t \to \infty$. That is, for any sequence $(t_n)_{n \in \N} \subseteq \R_{\geq 0} $ such that $t_n \stackrel{n \to \infty}{\to} \infty$, the sequence $(f_{t_n})_{n \in \N}$ does not converge in $\Pp{2}(\R^d \times \R^d)$.
    \end{counterexample*}
    \begin{proof}
        Aiming for a contradiction, we assume that there exists $(t_n)_{n \in \N} \subseteq \R_{\geq 0}$ satisfying $t_n \stackrel{n \to \infty}{\to} \infty $ and $f_\infty \in \Pp{2}(\R^d \times \R^d)$ such that $f_{t_n} \stackrel{n \to \infty}{\to} f_\infty$ in $\Pp{2}(\R^d \times \R^d)$. In particular, we have 
        \begin{align}
            \sup_{n \in \N} \int  \abs{r}^{2} \, \mathrm{d} \msolt{t_n}{r,v} < \infty \, . \label{Eqmombound}
        \end{align}
        By using convergence of velocities in Theorem \ref{thm:gradHconvergence} and the Cauchy-Schwarz inequality, we see that 
        \begin{align*}
            \int r \cdot (v-\bar{v}) \, \mathrm{d}\msolt{t_n} {r,v} \stackrel{n \to \infty}{\to} 0 \, .
        \end{align*}
        Moreover, since the forces also converge to zero (see \eqref{eq:Forceconvergence}), it follows that 
        \begin{align}
            \Big| &\int r \cdot \int \nabla \mathcal{V}(r-r^\prime) \, \mathrm{d}\msolt{t_n}{r^\prime, v^\prime} \, \mathrm{d}\msolt{t_n}{r,v} \Big|^{2} \nonumber \\
            &\leq \Big( \int \abs{r}^{2} \, \mathrm{d}\msolt{t_n}{r,v} \Big) \Big( \int \Big| \int \nabla \mathcal{V}(r-r^\prime) \, \mathrm{d}\msolt{t_n}{r^\prime,v^\prime} \Big|^{2} \, \mathrm{d}\msolt{t_n}{r,v}\Big) \stackrel{n \to \infty }{\to} 0 \, . \label{EqVirialconv}
        \end{align}
        Proposition \ref{prop:derivPfunctional} and the condition $\nabla \mathcal{V}(x) \cdot x \leq 0$ for all $x \in \R^d$ imply that the map  $t \mapsto \Pfunc (f_t)$ is non-decreasing. Moreover, by combining the monotonicity of $\Pfunc$ with boundedness of the second moments \eqref{Eqmombound} and the upper bound $\Psi (x) \leq \norm{\psi}_\infty x^2/{2}$, we conclude that 
        \begin{align*}
            \sup\limits_{t\geq 0} \Pfunc (f_t) = \sup\limits_{n \in \N} \Pfunc (f_{t_n}) < \infty \, .
        \end{align*}
         Therefore, the monotone limit
        \begin{align*}
            \Pfunc_\infty \coloneq \lim\limits_{t \to \infty} \Pfunc (f_t) = \lim\limits_{n \to \infty} \Pfunc (f_{t_n})
        \end{align*}
        exists and is finite. By assumption, it holds that either $\Pfunc (f_0) > 0$ or $\Pfunc (f_0) = 0$ and $\frac{\mathrm{d}}{\mathrm{d}t} \Big|_{t = 0} \Pfunc (f_0) > 0$ and therefore there exists $\delta > 0$ such that $\Pfunc (f_t) > 0$ for all $t \in (0,\delta)$, which implies that $P_\infty > 0$. \\
        As a next step, we show that $f_\infty = \delta_{(0,\bar{v})} $. The first moments of $f_\infty$ are given by 
        \begin{align}
            \int v \, \mathrm{d}f_\infty(r,v) = \lim_{n \to \infty} \int v \, \mathrm{d}\msolt{t_n}{r,v} = \bar{v}  \text{ and } 
            \int r \, \mathrm{d}f_\infty(r,v) =\lim_{n \to \infty} \int r \, \mathrm{d}\msolt{t_n}{r,v} = 0  . \label{eq:momfinfty}
        \end{align}
        Moreover, since $\nabla \mathcal{V}$ is bounded, we have 
        \begin{align}
            \int &\int \nabla \mathcal{V}(r-r^\prime) \cdot (r-r^\prime) \, \mathrm{d}f_\infty (r,v) \, \mathrm{d}f_\infty(r^\prime, v^\prime) \nonumber \\
            &= \lim\limits_{n \to \infty} \int \int \nabla \mathcal{V}(r-r^\prime) \cdot (r-r^\prime) \, \mathrm{d}\msolt{t_n}{r,v} \, \mathrm{d}\msolt{t_n}{r^\prime, v^\prime} \nonumber \\
            &= \lim\limits_{n \to \infty} 2 \int r \cdot \int \nabla \mathcal{V}(r-r^\prime) \, \mathrm{d}\msolt{t_n}{r^\prime, v^\prime} \, \mathrm{d}\msolt{t_n}{r,v} = 0 \, . \label{eq:momentlimit}
        \end{align}
        In the last step, we have used the bound \eqref{EqVirialconv}. Due to $\nabla \mathcal{V}(x) \cdot x < 0$ for all $0 \neq x \in \R^d$, we see that the diagonal $\{ r^\prime = r \} $ has full $f_\infty \otimes f_\infty$ measure. In combination with \eqref{eq:momfinfty}, we conclude that $(p_r)_{\#} f_\infty = \delta_{0}$, where $p_r \colon \R^d \times \R^d \to \R^d $ denotes the projection onto the first component. Similarly, the velocity component satisfies
        \begin{align*}
            \int \abs{v-\bar{v}}^{2} \, \mathrm{d}f_\infty(r,v) = \lim\limits_{n \to \infty } \int \abs{v-\bar{v}}^{2} \, \mathrm{d}\msolt{t_n}{r,v} = 0 \, . 
        \end{align*}
        In summary, we have shown that $f_{t_n}  \stackrel{n \to \infty}{\to} \delta_{(0,\bar{v})}$ in $\mathcal{P}_2(\R^d \times \R^d)$. At this point, we reach the contradiction
        \begin{align*}
            0 < P_\infty = \lim_{n \to \infty} \Pfunc (f_{t_n}) = \Pfunc (\delta_{(0,\bar{v})}) = 0 \, . \qquad \qedhere
        \end{align*}
    \end{proof}
    In the previous counterexample, we have seen that relative compactness in $\Pp{2}$ can fail if the term 
    \begin{align*}
        \frac{1}{2} \int \int (r-r^\prime) \cdot \nabla \mathcal{V}(r-r^\prime) \, \mathrm{d} \msolt{t}{r,v} \, \mathrm{d}\msolt{t}{r^\prime, v^\prime}
    \end{align*}
    has a negative sign. However, if this term is nonnegative provided that $\int \abs{r}^2 \, \mathrm{d} \msolt{t}{r,v}$ is sufficiently large, then we can ensure that second moments remain bounded, see Theorem \ref{thm:asymptotics}. Before giving a proof, we start with a proposition.
    \begin{proposition}\label{prop:compactness}
        Let $\mathcal{F} \subseteq \Pp{2}(\R^d \times \R^d)$ satisfy 
        \begin{align*}
            \sup\limits_{\mu \in \mathcal{F}} \int \abs{(r,v)}^2 \, \mathrm{d}\mu(r,v) < \infty. 
        \end{align*}
        Then $\mathcal{F}$ is relatively compact in $\Pp{p} (\R^d \times \R^d)$ for all $p \in [1,2)$. 
    \end{proposition}
    \begin{proof}
        For $R > 0$, we have 
        \begin{align*}
            \sup\limits_{\mu \in \mathcal{F}} \int_{\abs{r,v}\geq R} \abs{(r,v)}^p \, \mathrm{d} \mu(r,v) \leq R^{p-2} \,   \sup\limits_{\mu \in \mathcal{F}} \int_{\abs{r,v}\geq R} \abs{(r,v)}^2 \, \mathrm{d} \mu(r,v)
            \stackrel{R \to \infty}{\to} 0 \, .
        \end{align*}
        Thus, precompactness of $\mathcal{F}$ in $\Pp{p}(\R^d \times \R^d)$ follows by using \cite[Proposition $2.2.3$]{panaretos2020invitation}. 
    \end{proof}

   \begin{proof}[Proof of Theorem \ref{thm:asymptotics}] 
         Fix $t \geq 0$. We can assume that $\norm{\charX{t}{\cdot,\cdot } }_{L^{2}} > r_0$ otherwise \eqref{eq:momboundattr} is satisfied due to $\norm{\psi}_\infty / {\underline{\psi}} \geq 1$. We define 
        \begin{align*}
            t^* \coloneq \inf \{ s \in [0,t] \mid \forall \tau \in [s,t] \ \colon \norm{\charX{\tau}{\cdot,\cdot}}_{L^2} > r_0 \}  \, . 
        \end{align*}
        We note that $\norm{\charX{t^*}{\cdot, \cdot}}_{L^2} \leq \max\{ r_0, \norm{\charX{0}{\cdot, \cdot} }_{L^{2}} \} $. For $s \in [t^*, t]$, Proposition \ref{prop:derivPfunctional} yields 
        \begin{align}
            \dx{s}& \Big[ \int \charX{s}{r,v} \cdot (\charV{s}{r,v}-\bar{v}) \, \mathrm{d}\msolinit{r,v} \nonumber 
            + \frac{1}{2} \int \int \Psi \big(\abs{\charX{s}{r,v}- \charX{s}{r^\prime,v^\prime} } \big) \, \mathrm{d}\msolinit{r,v} \, \mathrm{d}\msolinit{r^\prime, v^\prime} \Big] \nonumber \\ 
            &= 
            - \frac{1}{2} \int \int \nabla \mathcal{V}(\charX{t}{r,v} - \charX{t}{r^\prime, v^\prime} ) \cdot (\charX{t}{r,v} - \charX{t}{r^\prime, v^\prime}  ) \, \mathrm{d}\msolinit{r,v} \, \mathrm{d}\msolinit{r^\prime, v^\prime} \nonumber \\
            &\quad + \int \abs{\charV{s}{r,v} - \bar{v}}^{2} \, \mathrm{d}\msolinit{r,v} \nonumber  \\
            &\leq \norm{\charV{s}{\cdot,\cdot} - \bar{v}}^{2}_{L^{2}} \, . \label{eq:pderivest}
        \end{align}
        For arbitrary $s \in [t^* , t]$, we obtain 
        \begin{align*}
            \dx{s} \frac{1}{2} \norm{\charX{s}{\cdot, \cdot} }_{L^{2}}^{2} 
            + \frac{1}{2} \underline{\psi} \norm{\charX{s}{\cdot, \cdot} }_{L^{2}}^{2}  
            \leq \Pfunc(f_s)
            \leq 
            \Pfunc (f_{t^*}) + \int_{t^{*}}^{s} \norm{\charV{\tau}{\cdot,\cdot} - \bar{v}}^{2}_{L^{2}} \, \mathrm{d}\tau \, , 
        \end{align*}
        where the second inequality above follows by integrating \eqref{eq:pderivest} from $t^*$ to $s$. Gronwall's lemma implies that 
        \begin{align}
            \frac{1}{2} \norm{\charX{s}{\cdot, \cdot} }_{L^{2}}^{2} &\leq \frac{1}{2} \norm{\charX{t^*}{\cdot, \cdot} }_{L^{2}}^{2}
            \mathrm{e}^{-\underline{\psi} (s-t^*)} 
            + \int_{t^*}^{s} \mathrm{e}^{-\underline{\psi} (s-\tau)} \Pfunc (f_{t^*}) \, \mathrm{d}\tau \nonumber \\
            &\quad + \int_{t^*}^{s} \mathrm{e}^{-\underline{\psi} (s-\tau)} \int_{t^{*}}^{\tau} \norm{ \charV{\sigma}{\cdot,\cdot} - \bar{v}}^{2}_{L^{2}} \, \mathrm{d}\sigma \, \mathrm{d}\tau \nonumber \\
            &= \frac{1}{2} \norm{\charX{t^*}{\cdot, \cdot} }_{L^{2}}^{2} \mathrm{e}^{-\underline{\psi} (s-t^*)}
            + \Pfunc (f_{t^*}) \frac{1-\mathrm{e}^{-\underline{\psi}(s-t^*) }}{\underline{\psi}} \nonumber \\ 
            &\quad +  \int_{t^*}^{s} \frac{1-\mathrm{e}^{-\underline{\psi}(s-\sigma) }}{\underline{\psi}} \norm{\charV{\sigma}{\cdot,\cdot} - \bar{v}}^{2}_{L^{2}} \, \mathrm{d}\sigma \nonumber \\
            &\leq \frac{1}{2} \norm{\charX{t^*}{\cdot, \cdot} }_{L^{2}}^{2} + \frac{1}{\underline{\psi}} \Pfunc (f_{t^*}) + \frac{1}{\underline{\psi}} \int_{t^*}^{s} \norm{\charV{\sigma}{\cdot,\cdot} - \bar{v}}^{2}_{L^{2}} \, \mathrm{d}\sigma \nonumber \\
            &\leq \frac{1}{2} \max \Big\{ r_0^{2}, \int \abs{r}^2 \, \mathrm{d}\msolinit{r,v} \Big\} 
            + \frac{1}{\underline{\psi}} \Pfunc (f_{t^*}) + \frac{1}{\underline{\psi}} \int_{t^*}^{s} \norm{\charV{\sigma}{\cdot,\cdot} - \bar{v}}^{2}_{L^{2}} \, \mathrm{d}\sigma \label{eq:RsL2} \, .
        \end{align}
        For the second and third term in \eqref{eq:RsL2}, we have the upper bounds
        \begin{align}
            \Pfunc (f_{t^*}) &= \int \charX{t^*}{r,v}  \cdot ( \charV{t^*}{r,v} - \bar{v}) \, \mathrm{d}\msolinit{r,v} \nonumber \\
            &\quad + \frac{1}{2} \int \int  \Psi \big(\abs{ \charX{t^*}{r,v} - \charX{t^*}{r^\prime, v^\prime} } \big) \, \mathrm{d}\msolinit{r,v} \, \mathrm{d}\msolinit{r^\prime,v^\prime} \nonumber \\
            &\leq \norm{\charX{t^*}{\cdot, \cdot} }_{L^{2}}  \norm{\charV{t^*}{\cdot, \cdot} - \bar{v}}_{L^{2}} 
            + \frac{\norm{\psi}_\infty}{2} \norm{\charX{t^*}{\cdot, \cdot} }_{L^{2}}^{2} \nonumber \\
            &\leq \max \big\{ r_0, \norm{\charX{0}{\cdot, \cdot}}_{L^{2}} \big\} \big( 2 (\mathcal{H}(f_{t^*}) - \underline{\mathcal{V}}) \big)^{\frac{1}{2}} 
            + \frac{\norm{\psi}_\infty}{2} \max \big\{ r_0^{2}, \norm{\charX{0}{\cdot, \cdot}}_{L^{2}}^{2} \big\} \nonumber \\
            &\leq \max \Big\{ r_0, \Big( \int  \abs{r}^2 \, \mathrm{d}\msolinit{r,v} \Big)^{\frac{1}{2} } \Big\} \big( 2(\mathcal{H}(f_{0}) - \underline{\mathcal{V}}) \big)^{\frac{1}{2}} 
             + \frac{\norm{\psi}_\infty}{2} \max \Big\{ r_0^{2}, \int \abs{r}^2 \, \mathrm{d}\msolinit{r,v} \Big\}  \, , \label{eq:Pftbound} \\  
        \int_{t^*}^{s} &\norm{\charV{\sigma}{\cdot,\cdot} - \bar{v}}^{2}_{L^{2}} \, \mathrm{d}\sigma 
            \leq \int_{t^*}^s - \frac{1}{\underline{\psi}} \dx{\sigma} \mathcal{H}(f_\sigma) \, \mathrm{d}\sigma 
            = \frac{1}{\underline{\psi}} (\mathcal{H}(f_{t_*}) - \mathcal{H}(f_s) ) 
            \leq \frac{1}{\underline{\psi}} (\mathcal{H}(f_{0}) - \underline{\mathcal{V}} ) \, . \label{eq:kinbound}
        \end{align}
        By combining the inequalities \eqref{eq:RsL2}, \eqref{eq:Pftbound} and \eqref{eq:kinbound}, we arrive at the desired bound \eqref{eq:momboundattr}. \\ 
        In summary, we have now shown that
        \begin{align*}
            \sup_{t \geq 0} \int \abs{r}^2 \, \mathrm{d} \msolt{t}{r,v} < \infty \, .
        \end{align*}
        If we additionally apply Lemma \ref{LVbound}, we see that the second moments of $(f_t)_{t \geq 0}$ are bounded uniformly with respect to $t$ and therefore Proposition \ref{prop:compactness} implies that $(f_t)_{t \geq 0}$ is precompact in $\Pp{p}$ for all $p \in [1,2)$. By applying Theorem \ref{thm:gradHconvergence}, we see that any $\Wpdist_p$-limit point $f_\infty \in \Pp{p}(\R^d \times \R^d)$ of $(f_t)_{t \geq 0}$ satisfies 
        \begin{align*}
            \int \Big| \int \nabla \mathcal{V}(r-r^\prime) \, \mathrm{d}f_\infty(r^\prime, v^\prime) \Big|^2 \, \mathrm{d}f_\infty (r,v) &= 0 \, , 
            \text{ and } \int \abs{v-\bar{v}}^2 \, \mathrm{d}f_\infty(r,v) = 0 \, .
        \end{align*}
        Moreover, since the second moments are lower semicontinuous with respect to weak convergence (see \cite[Lemma $12.8$]{Ambrosio}) and therefore also w.~r.~t.~ $\Wpdist_p$-convergence, we obtain that 
        \begin{align*}
            \int \abs{(r,v)}^2 \, \mathrm{d}f_\infty(r,v) \leq \liminf_{t \to \infty } \int \abs{(r,v)}^2 \, \mathrm{d} \msolt{t}{r,v} < \infty, 
        \end{align*}
        and therefore $f_\infty \in \Pp{2}(\R^d \times \R^d)$. In summary, we have shown that every $\Wpdist_p$-limit point of the $\Wpdist_p$-precompact family $(f_t)_{t \geq 0}$ belongs to $\mathcal{L}$, which implies the desired convergence \eqref{eq:Wpdistconv}. 
    \end{proof}

     \begin{remark}
         If the binary interactions are strictly attractive, that is, if $\nabla \mathcal{V}(x) \cdot x \geq 0$ for all $x \in \R^d $, then \eqref{eq:attrcase} holds with $r_0 = 0$ for every $f_0 \in \Ppo{\infty}(\R^d \times \R^d)$. \\
        On the particle level, that is, when $f_0$ is of the form 
            \begin{align*}
                f_0 = \masssum{j}{N} \delta_{(r_j, v_j)}
            \end{align*}
            for some $N \in \N$ and $(r_j, v_j) \in \R^d \times \R^d, \, j = 1, \ldots, N$, and if, in addition
            \begin{align*}
                \lim\limits_{\abs{x} \to \infty} \nabla \mathcal{V}(x) \cdot x = \infty,
            \end{align*}
             then \eqref{eq:attrcase} holds for sufficiently large $r_0 = r_0(N)$.
    \end{remark}

    We conclude this section by illustrating condition \eqref{eq:attrcase} for the (regularized) Morse potential.
    
    \begin{example}\label{ex:morsecounterex}
        Consider the (regularized) Morse potential 
        \begin{align}
            \mathcal{V}(x) = R \mathrm{e}^{-\abs{x}^2 / {r} } - A \mathrm{e}^{- \abs{x}^2 / a} , \text{ where } R \geq 0, \,  r,a,A > 0 \text{ and } a > r \, . \label{eq:regMorse}
        \end{align}
        If $\frac{a R}{r A} \leq 1$, then the interactions are strictly attractive, that is $\nabla \mathcal{V}(x) \cdot x \geq 0$ for all $x \in \R^d$. However, if $\frac{a R}{r A} > 1$, then \eqref{eq:attrcase} is not satisfied for a three particle initial condition 
        \begin{align*}
            f_0 = \masssum{j}{3} \delta_{(r_j, v_j)}, \, \text{ with } (r_j, v_j) \in \R^d \times \R^d \text{ mutually distinct.}  
        \end{align*} 
        To see this, we note that 
        \begin{align*}
            \nabla \mathcal{V}(x) \cdot x = 2 \abs{x}^2 \Big( \frac{A}{a} \mathrm{e}^{ - \abs{x}^2 /{a}} - \frac{R}{r} \mathrm{e}^{-\abs{x}^2/{r} } \Big) &< 0, \text{ for } 0 < \abs{x}^2 < \frac{ra}{a-r} \ln \Big( \frac{aR}{rA} \Big) = : d_0^2  , \\
            \text{ and } \lim\limits_{\abs{x} \to \infty} \nabla \mathcal{V}(x) \cdot x &= 0 \, .
        \end{align*}
        We fix $x,\nu \in \R^d$ with $\abs{\nu} = 1$ and $\abs{x}^2 \in (0,d_0^2 ) $ and for $n \in \N$, we define $R_n \in L^2(f_0;\R^d )$ by 
        \begin{align*}
            R_n(r_1, v_1) \coloneq n \nu, \, R_n(r_2, v_2) \coloneq n \nu + x, R_n(r_3,v_3) \coloneq - 2n \nu - x \, .  
        \end{align*}
        Then 
        \begin{align*}
            \int R_n(r^\prime, v^\prime) \, \mathrm{d} \msolinit{r^\prime, v^\prime} = 0, \text{ and } \norm{R_n}_{L^2(f_0)} \stackrel{n \to \infty}{\to} \infty \, .
        \end{align*}
        In addition, for $i \neq j$ with $\{ i, j \} \neq \{ 1,2 \} $, it holds that 
        \begin{align*}
            \abs{R_n(r_i,v_i) - R_n(r_j,v_j) } \stackrel{n \to \infty}{\to} \infty 
            \text{ while } R_n(r_1,v_1) - R_n(r_2,v_2)  = x \, .
        \end{align*}
        Therefore, 
        \begin{align}
            &\int \int \nabla \mathcal{V}(R_n(r,v) - R_n(r^\prime, v^\prime)) \cdot (R_n(r, v) -R_n(r^\prime, v^\prime)) \, \mathrm{d} \msolinit{r,v} \, \mathrm{d}\msolinit{r^\prime, v^\prime} \label{eq:virintegral} \\ 
            &= \frac{1}{9} \sum_{i, j = 1}^3  \nabla \mathcal{V} (R_n(r_i,v_i) - R_n(r_j, v_j)) \cdot (R_n(r_i, v_i) - R_n(r_j, v_j)) \stackrel{n \to \infty}{\to } \frac{2}{9} \nabla \mathcal{V}(x) \cdot x < 0 . \nonumber
        \end{align}
        In particular, the term \eqref{eq:virintegral} is eventually negative, which shows that \eqref{eq:attrcase} is not satisfied.
    \end{example}

    \section{Numerical results} \label{sec:numerics}
In the present section, we further assess Conjecture \ref{con:boundedness} by providing numerical computations for the (regularized) Morse potential \eqref{eq:regMorse} and different initial conditions.
In all examples, we use a constant alignment function $\psi(x) = 1/{2}, \, x \geq 0$ and we integrate the ODE system \eqref{eq:ODEintro} in dimension $d=2$ using an eighth order Runge--Kutta scheme to compute the center of mass coordinates $r_i = x_i - \bar{x}, \, i = 1, \ldots, N$. In all cases examined, the dynamics appear to become numerically stationary after a terminal time $T$, after which the simulation is stopped. In particular, this numerical behavior is consistent with the boundedness of trajectories asserted in Conjecture \ref{con:boundedness}. Below each figure, we provide the number of particles, the parameters of the interaction potential, the distribution of the initial conditions we sampled from and the terminal time $T$. The Python code can be found at \cite{Zenodofiles}.
Unlike in our original article \cite[Example $3.9$]{JaTo2024}, the (regularized) Morse potential \eqref{eq:regMorse} is in general not covered by the Stability Theorem \ref{thm:asymptotics} as shown in Example \ref{ex:morsecounterex}. Therefore, we focus our numerical analysis on this potential and we fix $A=r=1$. Then, if $aR \leq 1$, the regularized Morse potential is strictly attractive, otherwise it has a global minimum at $\abs{x}^2 = d_0^2 = a /{(a-1)} \ln(aR)$.  
Particles mutually repulse each other for interparticle distances less than $d_0$ and attract each other for larger distances.
To capture a range of behaviors that can occur for this potential, we consider the following different parameter regimes:
\begin{enumerate}
    \item Strong short-range repulsion, that is $1 \ll Ra$.
    \item Balanced attraction and repulsion, that is, $Ra$ is of order one.
   \item Strict attractivity, that is $Ra \leq 1$.
\end{enumerate}
In the first two regimes, we also vary the initial conditions to obtain different interparticle distances compared to the reference distance $d_0$. Numerical results for the strong short-range repulsion regime are shown in Figure \ref{fig:MorseHugeAlpha}.
\FloatBarrier
\begin{figure}[!htbp]
    \centering
    \begin{subfigure}[t]{0.48\linewidth}
        \includegraphics[width=\linewidth]{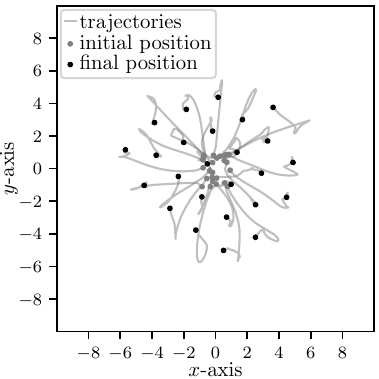}
        \caption{Initial positions at time $0$, final positions at time $T$ and the trajectories showing the movement of the particles from $0$ to $T$.}
        \label{fig:MorseHugeAlpha1}
    \end{subfigure}\ \
     \begin{subfigure}[t]{0.48\linewidth}
        \includegraphics[width=\linewidth]{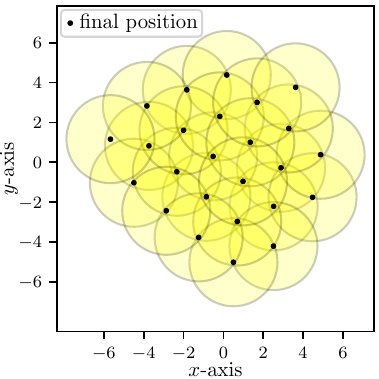}
        \caption{Final positions as in Figure \ref{fig:MorseHugeAlpha1}. Around each particle, a yellow disk of radius $d_0$ is shown.}
        \label{fig:MorseHugeAlpha2}
    \end{subfigure}
    \caption{Trajectories of the positions for the Morse potential \eqref{eq:regMorse} with dominant repulsion. The initial positions and velocities are sampled from the uniform distribution on $[-1,1]^2$. The parameters are $R=10.0$, $a=5.0$, $d_0=2.2$, $T=5000$.}
    \label{fig:MorseHugeAlpha}
\end{figure}
\FloatBarrier \noindent
The trajectories in Figure \ref{fig:MorseHugeAlpha1} show that the particles initially accelerate radially outwards due to their mutual repulsion. As time evolves, the particles decelerate due to the velocity alignment and the mutual long-range attraction. Finally, the particles settle into a tightly packed, lattice-like structure. In Figure \ref{fig:MorseHugeAlpha2}, disks of radius $d_0$ centered at the particle positions are superimposed on the final configuration. 
The distance between particles in the interior of the lattice is slightly smaller than $d_0$. Nevertheless, the total force acting on each particle vanishes. \\
Numerical results for the balanced repulsion and attraction regime with different initial conditions are shown in Figure \ref{fig:MorseCircle}.
\begin{figure}[!htbp]
    \centering
    \begin{subfigure}[t]{0.48\linewidth}%
    \centering
    \includegraphics[width=\linewidth]{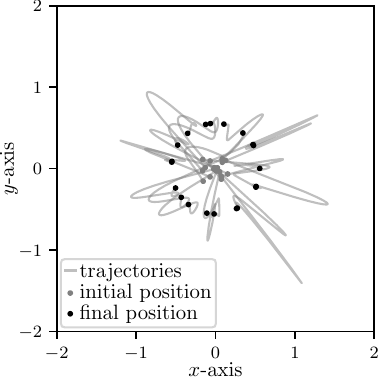}%
    \caption{Initial positions sampled from the uniform distribution on $[-d_0/(4\sqrt{2}),d_0/(4\sqrt{2})]^2$. Initial velocities sampled from the uniform distribution on $[-1,1]^2$.}%
        \label{fig:MorsecircleInitrepulsive}%
    \end{subfigure}\hspace{0.019\linewidth}%
     \begin{subfigure}[t]{0.48\linewidth}%
     \centering
        \includegraphics[width=\linewidth]{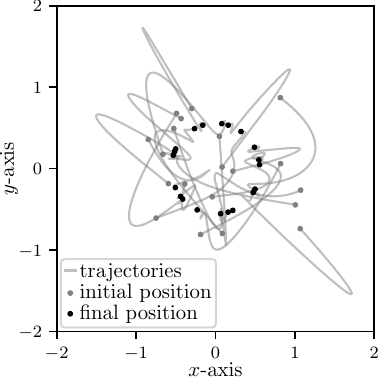}%
        \caption{Initial positions and velocities sampled from the uniform distribution on $[-1,1]^2$.}%
        \label{fig:MorsecircleInitBalanced}%
    \end{subfigure}
    \caption{Trajectories of the positions for the Morse potential \eqref{eq:regMorse} with balanced attraction and repulsion. The parameters are $N=20$, $R=0.4$, $a=5.0$, $d_0=0.93$, $T=600$.}
    \label{fig:MorseCircle}
\end{figure}
In Figure \ref{fig:MorsecircleInitrepulsive}, the initial conditions are chosen such that all interparticle distances are at most $d_0/2$, leading to mutual repulsion of all particles. This results in a rapid outward radial motion of the particles.
Due to the velocity alignment and the binary long-range attraction, the particles decelerate and then accelerate radially inward.
Subsequently, the particles oscillate around their final position before settling on a circle centered at the origin with radius $\approx 0.56$. 
An initial configuration, where both binary attraction and repulsion are present, is shown in Figure \ref{fig:MorsecircleInitrepulsive}. After an initial phase of complex irregular motion, the particles converge to a circular configuration centered at the origin with the same radius of approximately $0.56$ as in Figure \ref{fig:MorsecircleInitrepulsive}. 
For the same parameters as in Figure \ref{fig:MorseCircle}, convergence to a circular configuration also appears for more complex initial conditions. This is illustrated in Figure \ref{fig:morsenumerics:bigFourCenters}, where the initial positions are sampled from four distinct clusters.
\begin{figure}[!htbp]
    \centering
    \includegraphics[width=0.49\linewidth]{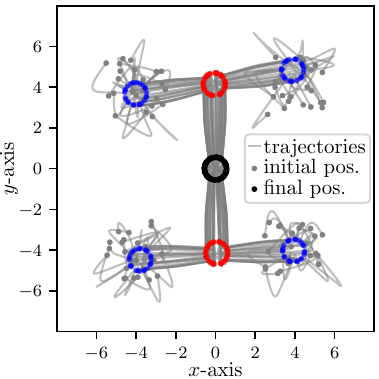}%
    \caption{Trajectories of the positions for the Morse potential \eqref{eq:regMorse}. The parameters are $N=100$, $R=0.4$, $a=5.0$, $d_0=0.93$, $T=50000$. Blue (resp.\, red) markers show positions at time $5000$ (resp.\, at time $25000$). The initial velocities are sampled from the uniform distribution on $[-1,1]^2$. For $s \in \{-4,4\}^2$, the initial positions of $25$ particles are sampled uniformly from $s+[-1.5,1.5]^2$.}
    \label{fig:morsenumerics:bigFourCenters}
\end{figure}
\FloatBarrier
    \begin{figure}[!htbp]
    \begin{subfigure}[t]{0.48\linewidth}%
        \centering%
        \includegraphics[width=\linewidth]{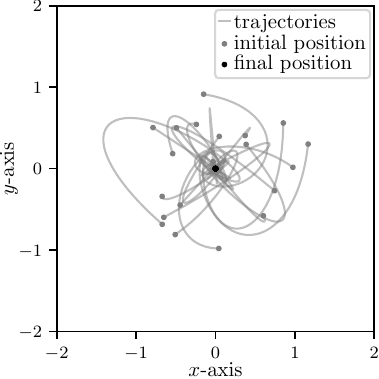}%
        \caption{The parameters are $N=20$, $R=0.2$, $a=2.5$, $T=600$.}%
        \label{fig:MorsecircleGlobalattractive}%
    \end{subfigure}\hspace{0.019\linewidth}%
    \begin{subfigure}[t]{0.48\linewidth}%
        \centering%
        \includegraphics[width=\linewidth]{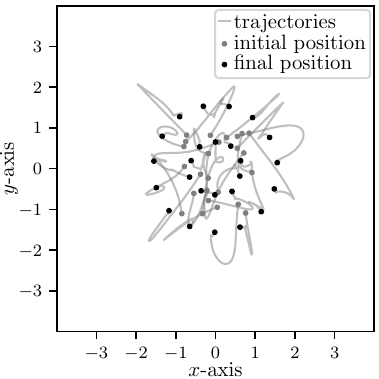}%
        \caption{The parameters are $N=25$, $R=1.5$, $a=5.0$, $d_0=1.6$, $T=2000$.}%
        \label{fig:MorseBigAlpha}%
    \end{subfigure}%
     \caption{Trajectories of the positions for the Morse potential \eqref{eq:regMorse}. The initial positions and velocities are sampled uniformly from $[-1,1]^2$.}%
     \label{fig:MorsePlotMultiple}%
\end{figure}
\FloatBarrier \noindent
The dynamics in Figure \ref{fig:morsenumerics:bigFourCenters} consists of three different phases. In the first phase, the particles settle towards a circular configuration (blue markers, time $t = 5000$) in each separate cluster. Subsequently, these four circles collapse into two distinct circles (red markers, $t = 25000$).
In the final phase, the two red circle collapse and form the familiar circular configuration centered at the origin with radius $0.56$ (at time $t = 50 000$).
To further explore the transition between the lattice-like (Figure \ref{fig:MorseHugeAlpha}) and circular final configuration (Figures \ref{fig:MorseCircle} and \ref{fig:morsenumerics:bigFourCenters}), we also simulate the intermediate repulsion strength $R = 1.5$, see Figure \ref{fig:MorseBigAlpha}. Here, the resulting final configuration consists of two concentric circles.
Finally, the results for the strictly attractive regime are shown in Figure \ref{fig:MorsecircleGlobalattractive}. The collapse of all particles to the origin in Figure \ref{fig:MorsecircleGlobalattractive} is in accordance with Theorem \ref{thm:asymptotics}. Indeed, due to $\nabla \mathcal{V}(x) \cdot x > 0$ for all $x \neq 0$, a similar argument as in the proof of Counterexample \ref{Cornoncompactness} shows that every element in the set $\mathcal{L}$ is a Dirac measure.
    
    \section*{Acknowledgments}
    This work was funded by the Deutsche Forschungsgemeinschaft (DFG, German Research Foundation) – Project-ID 531152215 – CRC 1701. We acknowledge the assistance of ChatGPT-5 mini for language suggestions, in accordance with the SIAM editorial policy.

        \putbib[arxivreferences]
    \end{bibunit}

    \clearpage
\setcounter{page}{1}
\setcounter{section}{0}
\setcounter{equation}{0}
\setcounter{figure}{0}
\setcounter{table}{0}
\stepcounter{partcounter}

\begin{center}
    {\LARGE Port-Hamiltonian structure of interacting particle systems and its mean-field limit \par}
    \vspace{0.5cm}
    {\large Birgit Jacob\footnote{Research Group Functional Analysis, \href{mailto:bjacob@uni-wuppertal.de}{bjacob@uni-wuppertal.de}} ,
Claudia Totzeck\footnote{Research Group Optimization, \href{mailto:totzeck@uni-wuppertal.de}{totzeck@uni-wuppertal.de}  } \par}
    \vspace{0.5cm}
    {IMACM, School of Mathematics and Natural Sciences, \\ University of Wuppertal, Germany \par}
    \vspace{0.5cm}
    {July 2024 \par}
\end{center}
\vspace{0.8cm}
\begin{abstract}
We derive a minimal port-Hamiltonian formulation of a general class of interacting particle systems driven by alignment and potential-based force dynamics which include the Cucker-Smale model with potential interaction and the second order Kuramoto model. The port-Hamiltonian structure allows to characterize conserved quantities such as Casimir functions as well as the long-time behaviour using a LaSalle-type argument on the particle level. It is then shown that the port-Hamiltonian structure is preserved in the mean-field limit and an analogue of the LaSalle invariance principle is studied in the space of probability measures equipped with the 2-Wasserstein-metric. The results on the particle and mean-field limit yield a new perspective on uniform stability of general interacting particle systems. Moreover, as the minimal port-Hamiltonian formulation is closed we identify the ports of the subsystems which admit generalized mass-spring-damper structure modelling the binary interaction of two particles. Using the information of ports we discuss the coupling of difference species in a port-Hamiltonian preserving manner.
\end{abstract}

\begin{minipage}{0.9\linewidth}
 \footnotesize
\textbf{AMS classification:} 37K45, 82C22, 93A16.
\medskip

\noindent
\textbf{Keywords:} Port-Hamiltonian systems, interacting particle systems, mean-field limit, long-time behaviour
\end{minipage}
 \phantomsection   
    \begin{bibunit}[abbrv]
    \section{Introduction}
Since the seminal works by Reynolds \cite{reynolds1987flocks} in 1987, Vicsek, Czirók, Ben-Jacob, Cohen,  Shochet \cite{vicsek1995novel} in 1995, and Cucker, Smale  \cite{CuckerSmale} in 2007, mathematical modelling of interacting particle systems and the structural analysis of these models attracts the attention of researchers from applied mathematics. One of the fascinating aspects is that simple interaction rules imposed for the binary interaction of two particles lead to collective behaviour of the whole crowd. In fact, often slight parameter changes can turn dynamics of ordinary differential equations (ODEs), where the particles form rings that resemble the milling of birds into clumps \cite{Dorsogna2006self}. 

As the analysis of these pattern formation is difficult on the particle level, where the position and velocity information of each member of the crowds is explicitly captured by the equations, more abstract formulations of the dynamics were proposed. Sending the number of particles to infinity leads to the so-called mean-field formulation of the crowd \cite{golse2016dynamics}. Here, the exact position and velocity information is averaged and only the probability of finding a particle at a certain time in a certain position with a certain velocity is described. The binary interaction structure on the particle level turns into an convolution that yields a nonlinear, nonlocal partial differential equation (PDE) as evolution equation on the mean-field level. Instead of a huge system of ODEs the mean-field equation requires the solution of a high-dimensional PDE. To further reduce the dimensionality of the problem, also hydro-dynamic descriptions were proposed, there the velocity information is averaged over space, leading to a coupled PDE-system that is only space-dependent \cite{hydro}. In contrast to the passage from mean-field level to hydro-dynamic formulations, which in general require formal closure relations, the rigorous relationship of particle and mean-field limit is well-understood, see for example \cite{golse2016dynamics}. For an detailed overview of different models on the various scales we refer to \cite{Shvydkoy}.

Many interacting particle systems are driven by two mechanisms: alignment in velocity and attraction/repulsion in space \cite{carrillo2017review}. Alignment is the main component of (generalized) Cucker-Smale dynamics and leads to bird-like behaviour. Many different interaction kernels were proposed to fit the model closer to reality, see for example \cite{ahn2012collision,carrillo2017sharp,cucker2011general,motsch2011newmodel}. 
On the other hand, there are models that incorporate only attraction/repulsion forces leading to a three-phase interaction behaviour of long-range attraction, short-range repulsion and a mid-range comfort zone, where no interaction forces are present \cite{albi2013modeling,Dorsogna2006self,Cao2020MBE}. These models often admit a gradient-structure, that means the binary interaction forces are gradients of a prescribed interaction potential like the Morse potential \cite{hydro} and the force acting on one particle is the average of all these binary interaction forces.

The gradient-structure opens the toolbox of gradient flows and large deviation theory to study stability and long-time behaviour of the systems \cite{ambrosio2005gradient}. In the absence of a gradient structure, the analysis of the alignment part of the dynamics is generally more complex. However, starting with \cite{carrillo2010asymptotic} stability results for different setups in the Cucker-Smale context were established \cite{motsch2014review,park2010cucker,ha2019complete,choi2016cucker,cho2016emergence,erban2016cucker,barbaro2016phase,pignotti2017convergence,albi2014stability}. In the literature stable states of alignment dynamics are often called flocking, clustering or consensus solutions. 

In \cite{Cao2020MBE} flocking behaviour of a three-zone model was first investigated with the help of the energy of the system. Although we address a similar question, in this contribution we propose a novel viewpoint on interacting particle dynamics, namely a port-Hamiltonian one. Building on the discussion of \cite{Matei}, where the port-Hamilto\-nian system (PHS) formulation of a Cucker-Smale dynamics with repulsion and attraction is interpreted as generalized mass-spring-damper system, we propose a minimal PHS representation of interacting particle systems that does not require an increase of the phase space. Indeed, the formulation in \cite{Matei} introduces the relative positions of all particles as new variables, thereby increasing the state space dimension from $2Nd$ to $N(N-1)d/2 + Nd$ where $N$ denotes the number of particles and $d$ the space dimension. In the following, the port-Hamiltonian reformulation preserves the state space dimension which allows in particular to pass to the mean-field limit. 

The port-Hamiltonian formulation of interacting particles opens the door to the well-established theory of PHS for finite-dimensional systems, see
\cite{vanDerSchaft06,EbMS07,DuinMacc09}.  Well-known facts from PHS theory include the property that PHS are closed under network interconnection, that is, coupling of port-Hamiltonian systems again leads to a port-Hamiltonian system. Furthermore, the port-Hamiltonian approach is suitable for the investigation of  the qualitative solution behavior such as asymptotic stability and control questions, as it provides an energy balance. Moreover, the port-Hamiltonian structure allows to identify several conserved quantities namely the Hamiltonian and the so-called Casimir functions. The mean-field limit yields a connection to nonlinear, nonlocal infinite-dimensional PHS systems which are yet less explored.

Our main contributions are: the port-Hamiltonian formulation of an interacting particle systems with same state space dimension; structure preserving mean-field limit; characterization of long-time behavior on the particle and mean-field level and the characterization of conserved quantities such as the Hamiltonian and Casimir functions. This yields a new perspective on uniform stability and the dissipativity of interacting particle systems using the PH dissipativity inequality. Moreover, the identification of ports allows for PH structure preserving coupling of interacting particle systems of (different) species.

The article is organized as follows: we recall some background information on interacting particle systems in Section~\ref{sec:background}. Then we motivate and derive the port-Hamiltonian reformulation in Section~\ref{sec:phsFormulation}, that is used to discuss the Casimir function and stability properties of the systems based on LaSalle theory. In Section~\ref{sec:mean_field} we discuss the mean-field limit sending the number of particles $N\rightarrow\infty$ and show that the PH structure is conserved in the limit. Section~\ref{sec:coupling} discusses the coupling of (different) species in a PH structure preserving manner. The article concludes with a summary of the main ideas and an outlook to future work.

\section{Background on interacting particle systems}\label{sec:background}

We recall the classical formulation of interacting particle systems in position and velocity coordinates and the corresponding mean-field limit. Let us consider $N\in\N, N\ge2$ interacting particles in space dimension $d.$ We denote their positions by $x_i \colon [0,T]\rightarrow \R^d$ and their velocities by $v_i \colon [0,T]\rightarrow \R^d$ for $i=1,\dots,N,$ respectively. We collect the position and velocity information of all particles in the vectors $x=(x_i)_{i=1}^N$ and $v=(v_i)_{i=1}^N,$ respectively. The dynamics of the $i$-th particle is given by
\begin{subequations} \label{eq:xvparticles}
\begin{align}
    \frac{\dd}{\dd t} x_i &= v_i, \\
    \frac{\dd}{\dd t} v_i &= \frac{1}{N}\sumj \psi(|x_j - x_i|)(v_j - v_i) - \frac{1}{N} \sumj \nabla \V(x_i - x_j)=:F_i(t,x,v), \\
    x_i(0) &= \hat x_i, \qquad v_i(0) = \hat v_i.
\end{align}
\end{subequations}
Here, $\psi \colon \R \rightarrow \R_{\ge 0}$ models the strength of the velocity alignment and $\V \colon \R^d  \rightarrow \R$ denotes the potential modelling the binary interactions among the particles.  
For the forces resulting from the interactions we require that $\V$ is continuously partial differentiable satisfying  
\begin{equation}\label{eq:antisymm}
\nabla \V(x) = - \nabla \V(-x), \qquad x\in \R^d.
\end{equation}
This general class of interaction models contains  well-known examples:
\cite{CuckerSmale}, \cite{Matei},  Morse-interactions (sheep flocks, double and single milling birds) as proposed in \cite{Dorsogna2006self}, or herding dynamics \cite{Totzeck}. 

To obtain the existence and uniqueness of a global solution to the particle system by standard results from ODE theory, we make the following
\begin{assumption}\label{ass:particle}
\begin{itemize}
    \item [(1)] $F_i$ is continuous on $[0,T] \times \R^{dN} \times R^{dN}$ for all $i\in\{1,\dots,N\}$
    \item [(2)] For some $C>0$ it holds
    \[|F_i(t,x,v)| \le C(1+|x|+|v|) \quad \text{ for all }\quad i\in\{1,\dots,N\}, t\in[0,T] \text{ and } x,v\in \R^{dN}\]
    \item [(3)] For all $i\in\{1,\dots,N\}$, $F_i(t,x,v)$ is locally Lipschitz continuous w.r.t.~$x$ and $v$. In particular, for every compact set $K \subset \R^{dN} \times \R^{dN}$ there exists some $L_K>0$ such that 
    $F_i$ is on $K$ Lipschitz continuous  w.r.t.~$x$ and $v$ uniformly for all  $i\in\{1,\dots,N\}$.
\end{itemize}
\end{assumption}
\noindent
The following proposition ensures that \eqref{eq:xvparticles} admits unique solutions in $\calC^1([0,\infty),\R^{dN}\times\R^{dN})$ the set of all continuously differentiable functions $z:[0,\infty)\rightarrow \R^{dN}\times\R^{dN}$.

\begin{proposition}\label{prop:WellODE}\cite{teschl}
Let Assumption~\ref{ass:particle} hold, then for every initial condition $(\hat x, \hat v) \in \R^{dN}\times \R^{dN}$ there exists a unique global solution $(x,v) \in \calC^1([0,\infty),\R^{dN}\times\R^{dN})$ to  \eqref{eq:xvparticles}.
\end{proposition}
As we are interested in the limiting behaviour as $N\rightarrow \infty$ we introduce the \textit{Wasserstein} or \textit{Monge-Kantorovich-Rubinstein} metric as distance measure between the particle and the density perspective. 
Let us denote by $\calP_2(\R^d)$ the space of Borel probability measures on $\R^d$ with finite $2$-nd moment. Equipping with the 2-Wasserstein distance, makes $\calP_2(\R^d)$ a complete metric space. We further denote $\calP_2^{ac}(\R^d)$ the subset of $\calP_2(\R^d)$ containing probability measures with Lebesgue density. For the sake of completeness we recall the 2-Wasserstein distance:
\[
W_2^2(\mu, \nu) := \inf\limits_{\pi \in \Pi(\mu, \nu)}\biggl\{\int_{\R^d} |x -y|^2 \dd\pi(x,y)\biggr\},\qquad \mu,\nu\in \calP_2(\R^d),
\]
where $\Pi(\mu,\nu)$ denotes the set of all Borel probability measures on $ \R^{2d}$ that have $\mu$ and $\nu$ as first and second marginals respectively, i.e.\
\[
\pi(B \times \R^d) = \mu(B), \qquad \pi(\R^d \times B) = \nu(B) \quad \text{for } B \in \mathcal{B}(\R^{d}).
\]

We emphasize that throughout the article we denote the integral of a function $\varphi \in \calC(\R^d)$ with respect to a probability measure $\mu \colon \R^d \rightarrow [0,1]$ by $\int \varphi(x) \mu(x) \dd x$, even if the probability measure is not absolutely continuous with respect to the Lebesgue measure, and hence does not have an associated density. Moreover, for evaluations of $f \in \calC([0,T],\calP(\R^d \times \R^d))$ at time $t$, position $x$ and velocity $v$ we write $f_t(x,v) := f(t,x,v)$. In particular, we use this for the empirical measure of the dynamics \eqref{eq:xvparticles}, which is given by
$$ \fN_t(x,v) := \sumi \delta(x - x_i(t)) \otimes \delta(v - v_i(t)). $$
It is well-known that $f^N$ satisfies the so-called mean-field equation
\begin{align}\label{eq:evolutionfN}
    &\partial_t \fN + \nabla_x \cdot (v\fN) = \nabla_v \cdot \Big( \int \psi(| y -  x |)(w - v) - \nabla \V(y - x) \dd f^N(t,y,w) \fN \Big)
\end{align}
in the weak sense. To be more precise, we consider the following notion of solution, where 
$\mathcal{C}_c^\infty$ denotes the set of all infinitely differentiable function with compact support.

\begin{definition}\label{def:weaksolution}
	We call $\mu\in \calC([0,T],\calP_2(\R^d\times \R^d))$ a {\em weak measure solution} of \eqref{eq:evolutionfN} with initial condition $\hat \mu \in \calP_2(\R^d\times \R^d)$ if and only if  for any test function $h\in \mathcal{C}_c^\infty([0, T)\times \R^d\times \R^d)$ we have
	\begin{align*}
	&\int_0^T \int_{\R^d} \big( \partial_t h_t  + v\cdot\nabla_x h_t \big)\dd \mu_t\dd t + \int_{\R^d} h_0\dd \hat \mu \\
	&= \int_0^T \int_{\R^d} \Big( \int \psi(| y -  x |)(w - v) - \nabla \V(y - x) \dd \mu_t(y,w) \cdot \nabla_v h\, \dd \mu_t\dd t.
	\end{align*}
\end{definition}
For notational convenience we introduce a short hand notation for the force term, let 
$$ F\ast\mu \colon \R^d \times \R^d \rightarrow \R^d, \quad \big(F\ast\mu\big)(x,v) = -\int \psi(| y -  x |)(w - v) - \nabla \V(y - x) \dd \mu(t,y,w). $$
For the empirical measure $f^N$ it holds 
$\big(F\ast f^N)(x_i,v_i) = F_i(x,v),$ where $F_i$ was already defined in \eqref{eq:xvparticles}. In particular, the assumption on $F_i$ in Assumption~\ref{ass:particle} imply the same properties for $F\ast\mu$ for compactly supported measures  $\mu \in \calP_c(\R^d\times\R^d).$ Moreover, we have the following well-posedness result on the mean-field level.

\begin{proposition}\label{prop:WellPDE}\cite{canizo2011well}
Let Assumption~\ref{ass:particle} hold and $f_0 \in \calP_c(\R^d\times\R^d).$ Then there exists a unique solution $f\in\calC([0,T],\calP_2(\R^d \times \R^d))$ to \eqref{eq:evolutionfN} in the sense of Definition~\ref{def:weaksolution} with initial condition $f_0$. Moreover, the solution remains compactly supported for all $t\in[0,T]$.
\end{proposition}

Moreover, in the limit $N\rightarrow\infty$ we have the well-known convergence of the empirical measure to the solution of the PDE in Wasserstein sense, see for example \cite{golse2016dynamics}.
\begin{proposition}[Dobrushin]\label{prop:Dobrushin}
     Let Assumption~\ref{ass:particle} hold, $\mu$ and $\mu'$ be solutions to the continuity equation \eqref{eq:evolutionfN} for compactly supported initial data $\hat \mu$, $\hat\mu' \in\calP_c(\R^d)$, respectively. Then, there exists a positive constant $c$ such that
 \begin{equation*}
 W_2^2(\mu_t,\mu_t') \le W_2^2(\hat \mu,\hat \mu') e^{c t}\quad \text{for all\, $t\in[0,T]$}.
 \end{equation*}
\end{proposition}

In the following we reformulate the ODE system as port-Hamiltonian system. The port-Hamilto\-nian structure opens the door for an alternative investigation of asymptotic flocking and uniform stability of general alignment-interaction models, which is discussed in the seminal article \cite{carrillo2010asymptotic} for the Cucker-Smale model. Moreover, we show that the port-Hamiltonian structure is preserved while passing to the mean-field limit and characterize Casimir functions of the dynamics.

\section{Port-Hamiltonian formulation }\label{sec:phsFormulation}

In this section, we derive two port-Hamiltonian formulations of the interacting particle system. We emphasize that both port-Hamiltonian formulations preserve the state space dimension. This is in contrast to \cite{Matei} where all relative positions between the particles are considered. First, we discuss the reformulation of the system given in $(x,v)$ coordinates. Then we present a variant that exploits the translational invariance of the systems.

Let us define $z=(x_1,\dots,x_N,v_1,\dots,v_N)$. The Hamiltonian of the system \eqref{eq:xvparticles} is  given by the sum of kinetic and potential energy
\[
\calH^N(z) = \frac12\sumi \Big(  v_i^\top v_i + \frac1N \sumj \calV(x_i - x_j) \Big),\quad z=(x_1,\dots,x_N, v_1,\dots, v_N).
\]
Then we calculate
\begin{align*}
\frac{\partial \calH^N(z)}{\partial z} = \left(-\frac{1}{N} \sumj \nabla \V(x_{j} - x_{1}), \ldots,  -\frac{1}{N} \sumj \nabla \V(x_{j} - x_{N}) ,v\right)^\top
\end{align*}
and obtain
\begin{equation}\label{eq:PHSparticlenew}
\frac{\dd}{\dd t} z = \left(\begin{pmatrix} 0 & I \\ -I & 0 \end{pmatrix}-\begin{pmatrix} 0 & 0 \\ 0 & \Psi(z) \end{pmatrix}\right) \frac{\partial \calH^N(z)}{\partial z}, \quad z(0)=z_0, 
\end{equation}
where
\[  \Psi(z) =\frac{1}{N}\begin{pmatrix} \displaystyle\sum_{j=2}^N \psi(|x_j-x_1|) &-\psi(|x_2-x_1|) & \ldots & -\psi(|x_N-x_1|)\\
-\psi(|x_1-x_2|) & \displaystyle\sum_{j=1,j\ne 2}^N \psi(|x_j-x_2|)& \ddots & \vdots\\
\vdots &\ddots &\ddots&&\\
&&&-\psi(|x_N-x_{N-1}|)\\
-\psi(|x_1-x_N|)&\ldots&-\psi(|x_{N-1}-x_N|)&\displaystyle\sum_{j=1}^{N-1} \psi(|x_j-x_N|)\end{pmatrix}.\]
The state space of the port-Hamiltonian system \eqref{eq:PHSparticlenew} is given by $X:=\R^{Nd+Nd}$.

\begin{remark}\label{rem:Psi}
For later use, we remark that  $\Psi(z)$ is diagonally dominant with non negative diagonal elements, hence positive semi-definite. Further, $\{ v \in \R^{Nd} : v = (\tilde v, \dots, \tilde v) \text{ for some } \tilde v \in \R^d \} $ is contained in the kernel of $\Psi(z)$. Moreover, if  $\psi(x)>0$  for $x\in\R$, then the eigenspace corresponding to the zero eigenvector is  spanned by $\{ v \in \R^{Nd} : v = (\tilde v, \dots, \tilde v) \text{ for some } \tilde v \in \R^d \} $.
\end{remark}
As the positions of the particles appear only relatively in the dynamics, interacting particle systems are translational invariant w.r.t.~$x$. This motivates us to consider the dynamics with center of mass shifted to zero. Let us therefore  define the center of mass by $\bar x = \frac1N \sumj x_j$ and consider $z=(r,v)=(r_1,\dots,r_N,v_1,\dots,v_N)$ with $r_i = x_i - \bar x$ being the position relative to the center of mass. Note that the velocity of the center of mass is conserved. Indeed, using 
$\nabla \mathcal V(-r) = -\nabla \mathcal V(r)$, we find
\begin{align}\label{eqn:meanv}
\sum_{i=1}^N\frac{\dd}{\dd t} v_i &= \frac{1}{N}\sum_{i=1}^N\sumj \psi(|r_j - r_i|)(v_j - v_i) - \frac{1}{N} \sum_{i=1}^N\sumj \nabla \V(r_i - r_j)=0,
\end{align}
Let us denote $\bar v := \frac1N \sumj v_j(0)$. The shifted dynamics is given by
\begin{subequations}\label{eq:phs}
\begin{align}\label{eqn:phs1}
    \frac{\dd}{\dd t} r_i &= v_i - \bar v,  &r_i(0) = (r_0)_i, \\
    \frac{\dd}{\dd t} v_i &= \frac{1}{N}\sumj \psi(| r_i - r_j |)(v_j - v_i) - \frac{1}{N} \sumj \nabla \mathcal V(r_ i - r_j),  &v_i(0)=(v_0)_i.\label{eqn:phs2}
\end{align}
\end{subequations}

We want to emphasize that system \eqref{eq:xvparticles} contains explicit position information of the particles, which we loose in the PHS formulation as we shift by the center of mass. However, pattern formation and uniform stability which are of interest in the context of interacting particle systems are translational invariant, hence the explicit position information plays a minor role. 

The mean velocity $\bar v$ can be incorporated in the Hamiltonian leading to
\[\calH^N(z) = \frac{1}{2}\sumi \Big((v_i-\bar v)^\top (v_i - \bar v) + \frac{1}{N} \sum_{j=1}^{N}  \V(r_i-r_j) \Big)
\]
with partial derivative
\begin{align*}
\frac{\partial \calH^N(z)}{\partial z} = \left(-\frac{1}{N} \sumj \nabla \V(x_{j} - x_{1}), \ldots,  -\frac{1}{N} \sumj \nabla \V(x_{j} - x_{N}) ,v_1-\bar v, \dots, v_N - \bar v\right)^\top.
\end{align*}
For notational convenience, we define $v-\textbf{1}\bar v:=(v_1-\bar v, \dots, v_N - \bar v)^\top$.

To find the port-Hamiltonian structure we note that the upper-left part of the skew-symmetric matrix is predefined by the differential equation for $r$. Indeed, we obtain
\[
\frac{\dd}{\dd t} r = I (v-\textbf{1}\bar v).
\]
Thanks to the structure of $\Psi$, see Remark~\ref{rem:Psi}, it holds
\begin{equation}\label{eqn:null}
\Psi(z)\textbf{1}\bar v = 0, 
\end{equation}
and hence we obtain
\[ 
\frac{\dd}{\dd t} v = -I\frac{\partial \calH^N(z)}{\partial r} -\Psi(z)\frac{\partial \calH^N(z)}{\partial v}
\]
with
\[  \Psi(z) =\frac{1}{N}\begin{pmatrix} \displaystyle\sum_{j=2}^N \!\psi(|r_j-r_1|) &-\psi(|r_2-r_1|) & \ldots & -\psi(|r_N-r_1|)\\
-\psi(|r_1-r_2|) & \displaystyle\sum_{j=1,j\ne 2}^N \!\psi(|r_j-r_2|)& \ddots & \vdots\\
\vdots &\ddots &\ddots&&\\
&&&-\psi(|r_N\!-\!r_{N-1}|)\\
-\psi(|r_1-r_N|)&\ldots&-\psi(|r_{N-1}\!-\!r_N|)&\displaystyle\sum_{j=1}^{N-1} \!\psi(|r_j-r_N|)\end{pmatrix}.\]
Combining the two equations yields the port-Hamiltonian structure
\begin{equation}\label{eq:PHSparticle}
\frac{\dd}{\dd t} z = \left[ \begin{pmatrix} 0 & I \\ -I & 0 \end{pmatrix} - \begin{pmatrix} 0 & 0 \\ 0 & \Psi(z) \end{pmatrix} \right] \frac{\partial \calH^N(z)}{\partial z}. 
\end{equation}
The state space of the port-Hamiltonian system \eqref{eq:PHSparticle} is given by $X:=\R^{Nd+Nd}$.

The main results of the article are based on this formulation. The following theorem discusses the well-posedness and some characteristics of solutions to the port-Hamiltonian system.
\begin{theorem}\label{thm:existencenew}
Let Assumption~\ref{ass:particle}  hold, then for every initial condition $z_0=(r_0,v_0)\in \R^{Nd+Nd}$ the port-Hamiltonian system \eqref{eq:PHSparticle} possesses a unique global solution $z=(r,v)$ satisfying the dissipativity inequality
\begin{equation}\label{eqn:dissipnew}
\frac{\dd}{\dd t} \calH^N(z)= - \langle (v - \textbf{1}\bar v), \Psi(z) (v-\textbf{1}\bar v) \rangle \le 0,
\end{equation}
where 
\begin{equation*}
  \bar v:=\sum_{j=1}^N v_j(0). 
\end{equation*}
Further, the velocity of the center of mass is conserved, that is,
 $ \bar v=\sum_{j=1}^N v_j(t)$, $t\ge 0$. 
\end{theorem}
\begin{proof}
The existence of a unique global solution $z=(r,v)$ follows from Proposition \ref{prop:WellODE}.
The non-negativity of $\psi$ yields $\Psi(z)$ is real, symmetric and diagonally dominant, hence positive semi-definite.
Thus we compute for the solution $z$
\begin{align*}
    \frac{\dd}{\dd t}\calH^N(z) &= \left\langle \frac{\partial \calH^N(z)}{\partial z} , \frac{\dd}{\dd t} z \right\rangle = \left\langle \frac{\partial \calH^N(z)}{\partial z} , \left[ \begin{pmatrix} 0 & I \\ -I^\top & 0 \end{pmatrix} - \begin{pmatrix} 0 & 0 \\ 0 & \Psi(z) \end{pmatrix} \right] \frac{\partial \calH^N(z)}{\partial z}  \right\rangle\\
    &= - \langle v - \textbf{1}\bar v, \Psi(z) v - \textbf{1}\bar v\rangle \le 0.
\end{align*}
This shows the dissipativity inequality. 
Finally, let $z_0=(r_0,v_0)\in \R^{Nd}\times \R^{Nd}$. Then \eqref{eqn:meanv} implies that $\sum_{i=1}^N v_i$ is constant. This completes the proof.
\end{proof}

In the following we assume that Assumption~\ref{ass:particle}  is satisfied, and thus the port-Hamiltonian system \eqref{eq:PHSparticle} has for every initial condition a unique global solution. Next we investigate  conserved quantities of our port-Hamiltonian formulation of interacting particle systems. 

\begin{definition}[Definition 6.4.1 in \cite{schaft2000l2}]
A function $C:X\rightarrow \R$ that is partially differentiable is called a \emph{Casimir function} for the port-Hamiltonian system \eqref{eq:PHSparticle} if
\[  \left(\frac{\partial C}{\partial z}(z)\right)^\top \left[ \begin{pmatrix} 0 & I \\ -I & 0 \end{pmatrix} - \begin{pmatrix} 0 & 0 \\ 0 & \Psi(z) \end{pmatrix} \right]=0, \qquad z\in X.\]
\end{definition}

A Casimir function is a conserved quantity as for solutions $z$ we obtain
\[ \frac{\dd}{\dd t} C(z) = \left(\frac{\partial C}{\partial z}(z)\right)^\top \left[ \begin{pmatrix} 0 & I \\ -I & 0 \end{pmatrix} - \begin{pmatrix} 0 & 0 \\ 0 & \Psi(z) \end{pmatrix} \right]\frac{\partial \calH^N(z)}{\partial z} =0\]
independently of the Hamiltonian $\calH^N$.

\begin{proposition}
A function $C:X\rightarrow \R$ that is partially differentiable is  a Casimir function for the port-Hamiltonian system  \eqref{eq:PHSparticle} if and only if 
$ C(r,v) =  \gamma$
for some $ \gamma \in \R$.
\end{proposition}
\begin{proof}
In \cite{schaft2000l2} it is shown that a function $C:X\rightarrow \R$ that is partially differentiable is  a Casimir function for the port-Hamiltonian system \eqref{eq:PHSparticle} if and only if 
\[  \left(\frac{\partial C}{\partial z}(z)\right)^\top  \begin{pmatrix} 0 & I \\ -I & 0 \end{pmatrix} =0 \quad \text{and} \quad \left(\frac{\partial C}{\partial z}(z)\right)^\top   \begin{pmatrix} 0 & 0 \\ 0 & \Psi(z) \end{pmatrix}=0, \qquad z\in X.\]
This is equivalent to
\[   \frac{\partial C}{\partial r}(z)  =0 \quad \text{and} \quad  \frac{\partial C}{\partial v}(z)=0, \qquad z\in X,\]
 and thus the statement of the proposition follows.
\end{proof}

\begin{remark}
Note that the system matrices of the formulations in $(x,v)$ and $(r,v)$ coincide. This allows to conclude that the Casimir functions of the different formulations coincide as well.
\end{remark}

\begin{remark}
We want to emphasize that, in contrast to \cite{Matei}, we do not require a null space condition in order to define the Casimir function in neither of the two formulations. This is due to the different choice of the port-Hamiltonian formulation.
\end{remark}

We conclude this section with a sufficient condition for asymptotic stability. Here the function $\mathcal H^N$ serves as candidate for a suitable Lyapunov function. The following lemma will be useful.

\begin{lemma}\label{lem:invariantnew}
Let Assumption~\ref{ass:particle} hold, $\psi(x)>0$  for $x\in\R$ and we define 
 \[ \mathcal N:= \{z=(r,v)\in X\mid  - \langle v - \mathbf{1}\bar v, \Psi(z)( v - \mathbf{1}\bar v )\rangle = 0\}. \]
 Let $z:[0,\infty)\rightarrow X$, $z=(r,v)$, be a solution of \eqref{eq:phs} with $z(t)\in \mathcal N$ for $t\ge 0$. Then the function $v$ is constant,
 \[  v_i(t)=\frac{1}{N}\sumj v_j(0)\quad\text{and}\quad \sumj \nabla \mathcal V(r_ i(t) - r_j(t))=0, \qquad i=1,\ldots, N, \, t\ge 0.\]
\end{lemma}
\begin{proof}
 As $\psi(x)>0$  for $x\in\R$,  we obtain $ v_i(t)=v_j(t)$ for  $i,j=1, \ldots, N$ and $t\ge 0$. This together with Theorem \ref{thm:existencenew} implies that $v$ is constant and satisfies $v_i(t)=\frac{1}{N}\sumj v_j(0)$ for $ i=1,\ldots, N$, and $t\ge 0$. The remaining statement follows directly from equation \eqref{eq:phs}. 
\end{proof}

\begin{theorem}\label{prop:stabil}
Let Assumption~\ref{ass:particle} hold, $\psi(x)>0$  for $x\in\R$ and $\nabla \mathcal V$ bounded.
Then  for every initial condition $ z_0=(r_0,v_0)\in \R^{Nd}\times \R^{Nd}$ the corresponding solution $z$ of \eqref{eq:PHSparticle} satisfies
\[ \lim_{t\rightarrow \infty} \dist(z(t), L) =0, \]
where 
\[ L:= \Big\{ (\tilde r,\tilde v)\in \R^{Nd}\times \R^{Nd}\mid \tilde v_i=\frac{1}{N}\sumj v_j(0),\,\, \sumj \nabla \mathcal V(\tilde r_ i - \tilde r_j)=0, \, i=1,\ldots, N\Big\}. \]
\end{theorem}

\begin{proof}
Our goal is to apply LaSalle's stability theorem \cite[Theorem 3.2.11]{HiPr},  hence we have to show that the trajectories are contained in a compact subset of $\R^{Nd}\times \R^{Nd}$. 

Let $\bar v$ the mean velocity as above. We estimate
\begin{align*}
\| v(t) - \textbf{1}\bar v \|^2 &= \| v(0) - \textbf{1} \bar v\|^2 + \int_0^t 2 (v- \textbf{1}\bar v) \cdot \frac{d}{ds} (v- \textbf{1}\bar v) \dd s \\
&= \| v(0) - \textbf{1} \bar v\|^2 + \int_0^t 2 (v - \textbf{1}\bar v) (-\frac{\partial}{\partial r} \mathcal H^N(z) - \Psi(z) (v - \textbf{1}\bar v)) \dd s,
\end{align*}
where we used $\frac{d}{dt} \bar v = 0$ and $\Psi(z)\textbf{1}\bar v =0$. The boundedness of $\nabla \mathcal V$ and the Peter-Paul inequality allow us to estimate for $\varepsilon>0$
\begin{align*}
\| v(t) - \textbf{1}\bar v \|^2 &\le \| v(0) - \textbf{1} \bar v\|^2 + \int_0^t \frac{1}{\varepsilon} \| \nabla \mathcal V\|_\infty^2 +  \varepsilon \| v - \textbf{1}\bar v \|^2 - 2(v - \textbf{1}\bar v) \Psi(z)(v - \textbf{1}\bar v) \dd s \\
&\le  \| v(0) - \textbf{1} \bar v\|^2 +  \frac{t}{\varepsilon}  \| \nabla \mathcal V\|_\infty^2 - (2\lambda_2 - \varepsilon) \int_0^t \| v - \textbf{1}\bar v\|^2 \dd s,
\end{align*}
where we denote by $\lambda_2$ the second smallest eigenvalue of $\Psi(z)$. Note that $\lambda_2 >0$ since the $(n-1) \times (n-1)$ submatrix $\big(\Psi(z)_{ij}\big)_{i=1,\dots, n-1;j=1,\dots,n-1}$ is strictly diagonally dominant and therefore positive definite.

We define 
$$\alpha(t):= \| v(0) - \textbf{1} \bar v\|^2 +  \frac{t}{\varepsilon}   \| \nabla \mathcal V\|_\infty^2,$$
and then  Gronwall's inequality yields 
\[ \| v(t) - \textbf{1}\bar v \|^2  \le \alpha(t) e^{- (2\lambda_2 - \varepsilon)t}.\] 
Choosing $\varepsilon$ such that $2\lambda_2 > \varepsilon > 0$, we find that $v(t)$ relaxed towards $\textbf{1}\bar v$. In particular, $$\sup \{\| v(t) - \textbf{1}\bar v\|^2\mid t\ge 0\}<\infty,$$
which shows the boundedness of the solution trajectory w.r.t.~the velocity variable. 

We are left to show the boundedness of the trajectory w.r.t.~the positions. We estimate
\begin{align*}
\| r(t) \| \le \| r(0) \| + \int_0^t \|v(s) - \textbf{1}\bar v \| \dd s 
\le \| r(0) \| +  \int_0^\infty \alpha(s) e^{- (2\lambda_2 - \varepsilon)s}\dd s. 
\end{align*}
As the integral on the right hand side converges, we 
 conclude that also the position variables of the trajectories are bounded uniformly for all $t\ge 0$. Altogether, for any initial data the trajectories are contained in a compact set for all times. The statement follows by LaSalle's stability theorem.
\end{proof}

\begin{example}
Let us consider the example of the Cucker-Smale dynamics with potential, where the alignment function \cite{carrillo2010particle}  
\[
\psi(|r_j - r_i|) = \frac{K}{(\delta^2 + |r_j - r_i|^2)^\beta}, \qquad K,\delta >0 \text{ and }\beta \ge 0
\]
and (regularized) Morse potential \cite{Dorsogna2006self}
\[
\calV(d) = R e^{-|d|^2/r} - A e^{-|d|^2/a}, \qquad R,A \ge 0 \text{ and } r,a >0.
\]
are explicitly given. Note that $\psi \colon \R_{\ge 0} \rightarrow \R_{>0}$ is strictly positive and suppose that $r>a$ and $A>0$. 
 Then $\calH^N$ satisfies the assumption of  Theorem \ref{prop:stabil}. Thus  the Hamiltonian decreases as long as the velocities of the swarm members are not aligned. In particular, this yields unconditional flocking. Comparing our result to \cite{carrillo2010asymptotic} we find that the PHS structure allows us to boil the argument for flocking down to an application of LaSalle's stability theorem. However, we note that we require $\V\ne 0$ in the proof. Instead the argument in  \cite{carrillo2010asymptotic} exploits structure of the support of the solution to the particle system and holds for $\V \equiv 0$. 
 \end{example}

\begin{remark}
 Also the Kuramoto model with inertia and fully connected incidence matrix \cite{tanaka1997first}  fits into the setting with $\V(r_i - r_j) = -\cos(r_i -r_j)$ leading to
\begin{subequations} \label{eq:xvsimple}
\begin{align}
    \frac{\dd}{\dd t} r_i &= v_i-\bar v, \qquad i=1,\dots,N\\
    \frac{\dd}{\dd t} v_i &= -\gamma  (v_i-\bar v) + \frac{1}{N} \sumj \sin(r_i - r_j)\\ 
    r_i(0) &= \hat r_i, \qquad v_i(0) = \hat v_i.
\end{align}
\end{subequations}
with friction parameter $\gamma>0$, can be formulated as port-Hamiltonian system.  We obtain the  Hamiltonian 
\[
\calH^N(z) = \frac12\sumi \Big(  (v_i-\bar v)^\top (v_i - \bar v) - \frac1N \sumj \cos(r_i - r_j) \Big),\quad z=(r_1,\dots,r_N, v_1,\dots, v_N)
\]
and the dynamics
\[
\frac{\dd}{\dd t} z = (J-R) \frac{\partial \calH^N(z)}{\partial z}, \; z(0)=z_0, \quad J = \begin{pmatrix} 0 & I  \\ -I& 0 \end{pmatrix}, \; R = \begin{pmatrix} 0 & 0 \\ 0 & \gamma I \end{pmatrix}.
\]
As $R$ has higher rank compared to the one discussed above, the stability result with equilibrium point $v_i = 0$ is easier to show for this particular case of the Kuramoto model.
\end{remark}

\section{Mean-field limit}\label{sec:mean_field}
We obtain a candidate for the mean-field equation of the shifted dynamics following the standard derivation and moreover the mean-field Hamiltonian by rescaling with $\frac{1}{N}$. In fact, using the empirical measure $$f^N(t,r,v) = \frac{1}{N} \sumi \delta(r - r_i(t)) \otimes \delta(v-v_i(t)),$$ 
were $(r_i,v_i)_{i=1}^N$ denotes the solution of \eqref{eq:phs}, we obtain the PDE describing the mean-field dynamics given by
\begin{equation}\label{eq:evoFN}
\partial_t f^N + \nabla_r \cdot \Big( \big(v- \bar v  \big) f^N \Big) = \nabla_v \cdot \Big(\big( \int \psi(r- \hat r)(v- \hat v) + \nabla \calV(r - \hat r) \dd f^N(\hat r, \hat v) \big) f^N \Big),
\end{equation}
where $\bar v = \int \hat v \dd f^N(0, \hat r, \hat v)$. 
Moreover, for the mean-field Hamiltonian we obtain
\[\lim\limits_{N\rightarrow \infty} \frac{1}{N} \calH^N(z)  = \lim\limits_{N\rightarrow \infty} \int_{\mathbb R^d\times \mathbb R^d}\int_{\R^d \times \R^d} \frac{1}{2} \left(   (v-\bar v)^\top (v - \bar v) + \V(r-\hat r)\right) \dd f^N(t,\hat r,\hat v) \dd f^N(t,r,v).\]
This motivates to define the Hamiltonian for the mean-field equation as 
\begin{align*}\calH(f)  &=  \int_{\mathbb R^d\times \mathbb R^d}\int_{\R^d \times \R^d} \frac{1}{2} \left(    (v - \bar v)^\top (v - \bar v)  + \V(r-\hat r)\right) \dd f(t,\hat r,\hat v) \dd f(t,r,v) \\
&= \frac{1}{2}\int_{\mathbb R^d\times \mathbb R^d} f\ast \big(v^\top v + \mathcal V(r)\big) \dd f(t,r,v),
\end{align*}
where $\bar v = \int \hat v \dd f(t, \hat r, \hat v)$.

To check if the Hamiltonian structure is preserved in the mean-field limit we compute the variation of $\calH(f)$ in the space of probability measures, see also \cite{burger2021mean}. In order to preserve the normalization of the measures, we consider the push forward of $f$ w.r.t.~$(\text{id}+ \varepsilon \zeta)$ for $\zeta \in \calC_c(\mathbb R^{2d},\mathbb R^{2d})$, $\zeta(r,v)=(\zeta_r(r,v),\zeta_v(r,v)) $ and $\varepsilon > 0$ to find
\begin{align*}
&\calH((\text{id} +\varepsilon \zeta)_\#f) -\calH(f) \\
&
=\frac{1}{2} \int_{\mathbb R^d\times \mathbb R^d}\int_{\R^d \times \R^d}    \left(v+\varepsilon \zeta_v(r,v) - \bar v\right)^\top \left(v+\varepsilon \zeta_v(r,v) - \bar v\right)\dd f(t,\hat r,\hat v) \dd f(t,r,v)\\
&\quad +\frac{1}{2}\int_{\mathbb R^d\times \mathbb R^d}\int_{\R^d \times \R^d}  \V((r+\varepsilon \zeta_r(r,v))-\hat r-\varepsilon \zeta_{r}(\hat{r},\hat{v}))) \dd f(t,\hat r,\hat v) \dd f(t,r,v)\\
&\quad-\frac{1}{2}\int_{\mathbb R^d\times \mathbb R^d}\int_{\R^d \times \R^d} \left(   (v - \bar v)^\top (v - \bar v) + \V(r-\hat r)\right) \dd f(t,\hat r,\hat v) \dd f(t,r,v)
\\
&= \varepsilon \int_{\mathbb R^d\times \mathbb R^d}\int_{\R^d \times \R^d}  (v - \bar v)^\top \zeta_v(r,v) +\nabla \V(r-\hat r) \zeta_r(r,v) \, \dd f(t,\hat r,\hat v) \dd f(t,r,v) + o(\varepsilon),
\end{align*}
where we used \eqref{eq:antisymm} and Fubini's Theorem.
In the limit $\varepsilon\rightarrow 0$ we obtain 
\begin{equation*}
  \lim_{\varepsilon\rightarrow 0}  \frac{1}{\varepsilon} \Big(\calH((\text{id} +\varepsilon \zeta)_\#f) -\calH(f) \Big) = \int_{\mathbb R^d\times \mathbb R^d}\int_{\R^d \times \R^d} \begin{pmatrix}  \nabla \V(r-\hat r) \\ v - \bar v \end{pmatrix} \cdot \zeta(r,v) \dd f(t,\hat r,\hat v) \dd f(t,r,v).
\end{equation*}
Following \cite{ambrosio2005gradient} we can identify
\begin{equation}
\nabla_f \calH(f) = \begin{pmatrix}  
\left(\int_{\R^d \times \R^d} \nabla \V(r-\hat r) \, \dd f(t,\hat r,\hat v)\right)  \\  v-\bar v \end{pmatrix} = \begin{pmatrix}  
\left(\int_{\R^d \times \R^d} \nabla \V(r-\hat r) \, \dd f(t,\hat r,\hat v)\right)  \\ \int_{\R^d \times \R^d} (v-\hat v) \; \dd f(t,\hat r,\hat v) \end{pmatrix} = f\ast \begin{pmatrix} \nabla \V(r) \\ v \end{pmatrix}.
\end{equation}
We can rewrite the mean-field equation as
\begin{equation}\label{eq:meanfieldPHS}
 \partial_t f = -\nabla_{(r,v)} \cdot \left(f \left( f \ast \left(\left(\begin{pmatrix} 0 & I \\ -I & 0 \end{pmatrix} -\begin{pmatrix} 0 & 0 \\ 0 & \psi(r) \end{pmatrix}\right) \begin{pmatrix} \nabla \V(r) \\ v \end{pmatrix} \right) \right) \right).
\end{equation}

As expected we obtain the dissipativity inequality also on the mean-field level.
\begin{theorem}\label{thm:dissipativityF}
Let $\psi(x) > 0$ for all $x\in \R$ and  Assumption~\ref{ass:particle} hold. Then the dynamics \eqref{eq:meanfieldPHS} admits for every initial condition $f_0\in \calP_c(\R^d \times \R^d)$ a unique solution $f$ in the sense of Definition~\ref{def:weaksolution}. Moreover, the solution is dissipative, i.e.~it holds
\[
\frac{\dd}{\dd t} \calH(f) \le 0.
\]
Further, for every initial condition $f_0 \in \calP_c(\R^d \times \R^d)$ there exists a vector $\bar v \in \R^d$ such that 
\[
\int v\, \dd f(t,r,v) = \bar v \quad \text{for }t\ge 0.
\]
\end{theorem}
\begin{proof} 

We prove the last statement first. Indeed, using integration by parts we calculate 
\begin{align*}
\frac{\dd}{\dd t} \int v\, \dd f (t,r,v) 
&= \int v \, \nabla_v \cdot \Big( \int \psi(| r -  \hat r |)(v - \hat v) - \nabla \V(r - \hat r) \dd f(t,\hat r,\hat v) f \Big) \dd r \dd v \\
&= \iint \psi(|r - \hat r|) (v - \hat v) - \nabla \calV(r-\hat r) \dd f(t,\hat r, \hat v) \dd f(t,r,v) \\
&= -\frac{1}{2} \iint \nabla \calV(r-\hat r) \dd f(t,\hat r, \hat v) \dd f(t,r,v) + \frac{1}{2} \iint \nabla \calV(r- \hat r) \dd f(t,\hat r, \hat v) \dd f(t,r,v) \\
&= 0,
\end{align*}
which proves that the first moment with respect to the velocity is preserved. The well-posedness of \eqref{eq:meanfieldPHS} is obtained  with the same arguments as \eqref{eq:evolutionfN} in Proposition~\ref{prop:WellPDE}. 

For the dissipativity we use the product rule, the symmetry of $\psi$ and the anti-symmetry of $\nabla \calV$  to obtain  
\begin{equation*}
    \frac{\dd}{\dd t} \calH(f) 
    = -\frac12 \iint (v - \hat v) \cdot  ( \psi(|r - \hat r|)(v - \hat v)) \dd f(t,\hat r,\hat v) \dd f(t,r,v) 
    \le 0.\qedhere
\end{equation*}
\end{proof}

Next we investigate the characteristics and their port-Hamiltonian formulation. The characteristics $Z_t(z) = (R_t(r),V_t(v))$ read
\begin{subequations}\label{eq:PHSmf}
\begin{align}
\frac{\dd}{\dd t} R_t(r) &= V_t(v) - \int V_t(v) \dd f_0(r,v), \qquad \text{law}(z) = f(0,r,v),\\
\frac{\dd}{\dd t} V_t(v) &=  \int \psi(|R_t(r) - R_t(\hat r)|)( V_t(\hat v) - V_t(v)) - \nabla \mathcal V(R_t(r) - R_t(\hat r)) \,\dd f(0,\hat r,\hat v),
\end{align}
\end{subequations}
and we can rewrite a solution to the mean-field equation as $f(t) = Z_t{_\#} f(0),$ where $\#$ denotes the push-forward operator given by
\begin{equation*}
    \int_{\R^d \times \R^d} f(z)\, \dd f(t,z) = \int_{\R^d \times \R^d} f(Z_t(z))\, \dd f(0,z).
\end{equation*}

A natural question is if the Dobrushin inequality in combination with the stability result for the ODE dynamics yields a stability result on the mean-field level. This would require the interchangebility of the limits $N\rightarrow \infty$ and $t\rightarrow \infty$ which is beyond the scope of this article. However, we prove the stability on the mean-field limit with the help of LaSalle's theorem in metric spaces \cite[Theorem 3.2.11]{HiPr} in the following. Clearly, the precompactness argument is more involved in this setting. We begin with a the mean-field analogue of Lemma~\ref{lem:invariantnew}. For notational convenience we write in the following $f_t$ for $f(t,\cdot)$ as we already did in Section \ref{sec:background}.

\begin{lemma}\label{lem:invariantMF}
Let Assumption~\ref{ass:particle} hold, $\psi(x) >0$ for $x \in\R$ and 
\[
\mathcal N := \{ h \in \calP_2(\R^d \times \R^d)  \mid \iint (v-\hat v)^\top \psi(|r-\hat r|) (v-\hat v) \dd h(\hat r, \hat v) \dd h(r,v) = 0 \}.
\]
Let $f \colon [0,\infty) \rightarrow  \calP_2(\R^d \times \R^d) $ be a solution of \eqref{eq:meanfieldPHS} with $f_t \in \mathcal N$ for $t\ge 0$. Then the support of $f_t$ in the velocity space is concentrated at  $\bar v:= \int v \dd f_0(r,v)$, i.e.~there exists $g_t \in \calP_2(\R^d)$ such that it holds
\[
f_t(r,v) = g_t(r)\delta(v-\bar v) \quad \text{and} \quad \int \nabla \V(r-\hat r) \dd g_t(\hat r) = 0, \quad \text{for all } r \in \mathrm{supp}(g_t), t \ge 0.  
\]
\end{lemma}
\begin{proof}
As $\psi(x)>0$ for $x\in \R$, the support of $f_t$ in the velocity variable has to be concentrated at some velocity $\bar v$ for all $t\ge 0$ in order to satisfy the integration condition in the definition of $\mathcal N$. Together with Theorem~\ref{thm:dissipativityF} this implies that $\bar v = \int v \dd f_0(r,v)$. The remaining statement follows directly from equation~\eqref{eq:meanfieldPHS}. 
\end{proof}

\begin{theorem}\label{prop:stabilityMF}
Let Assumption~\ref{ass:particle} hold and there exists $\underline{\psi}$ with $\psi(x) \ge \underline{\psi}>0$ for $x\in\R$. Assume that $\calH$ is bounded from below and that $\nabla \V$ is bounded. Then for every initial condition $f_0 \in \calP_c(\R^d \times \R^d)$ the corresponding solution $f$ of \eqref{eq:meanfieldPHS} satisfies
\[
\lim\limits_{t\rightarrow \infty}\mathrm{dist}(f_t, \mathcal L) =0,
\]
where
\[
\mathcal L = \left\{ f \in \calP_2(\R^d \times \R^d) \mid f=g(x)\delta(v-\hat v), \quad\int \nabla \calV(x-\bar x) \dd g(x) = 0,\quad \hat v = \int v\, \dd f_0(x,v) \right\}.
\]
\end{theorem}

\begin{proof}
 The proof follows the lines of the finite-dimensional result. The aim is again to employ LaSalle's invariance theorem. Hence we have to show that the probability measure corresponding to the solution stays compactly supported for all times $t\ge 0$, since then the precompactness w.r.t.~$\mathcal W_2$ follows by \cite[Proposition 2.2.3]{panaretos2020invitation}. 

We first consider the 2-Wasserstein distance of $f_t$ and $\delta_{\bar v}$, where $\bar v = \int v \dd f_t(r,v)$, given by
\begin{align*}
\mathcal W_2&(f_t, \delta_{\bar v})^2\\ &= \int |v - \bar v|^2 \dd f_t = \int |v-\bar v|^2 \dd f_0 + \int_0^t \int  |v- \bar v|^2 \partial_s f_s \dd z \dd s \\
&=  \int |v-\bar v|^2 \dd f_0 - \int_0^t \int  |v- \bar v|^2 \cdot \nabla_{(r,v)} \cdot \left( f \ast \left(\begin{pmatrix} 0 & I \\ -I & 0 \end{pmatrix} -\begin{pmatrix} 0 & 0 \\ 0 & \psi(r) \end{pmatrix}\right) \begin{pmatrix} \nabla \V(r) \\ v \end{pmatrix} \right) \dd f_s \dd s \\
&= \int |v-\bar v|^2 \dd f_0 - 2 \int_0^t \int (f\ast v) (f\ast (\psi(r)v + f \ast \nabla \V(r) )) \dd f_s \dd s. 
\end{align*}
A simple computation shows that $(f\ast v) (f\ast \psi(r) v) = \frac12 f\ast (v\psi(r)v)$ which will be helpful in the following.

This and Peter and Paul inequality applied to the term with the potential interactions allows us to further estimate
\begin{align*}
\int |v - \bar v|^2 \dd f_t  &\le   \int |v-\bar v|^2 \dd f_0 - \int_0^t  \int \int (\underline{\psi} - \varepsilon) |v - \hat v|^2 \dd \hat f_s \dd f_s \dd s + \frac{\| \nabla \V \|_\infty^2}{\varepsilon} \\
&\le \int |v-\bar v|^2 \dd f_0 - \int_0^t   \int (\underline{\psi} - \varepsilon) |v - \bar v|^2 \dd f_s  \dd s + \frac{\| \nabla \V \|_\infty^2}{\varepsilon}.
\end{align*}
An application of Gronwall's inequality yields
\[
\mathcal W_2(f_t, \delta_{\bar v})^2 = \int |v - \bar v|^2 \dd f_t \le \left(\int |v-\bar v|^2 \dd f_0+\frac{\| \nabla \V \|_\infty^2}{\varepsilon} \right) e^{-(\underline{\psi} - \varepsilon)t}.
\]
This shows that the support of $f_t$ w.r.t.~$v$ strictly decays over time.

We use this result to show the boundedness of the support w.r.t.~$r$. Indeed, we obtain
\begin{align*}
\int |r|^2 \dd f_t &= \int |r|^2 \dd f_0 - 2 \int_0^t \int r  \cdot (v- \bar v) \dd f_s \dd s \\
&\le 1+\int |r|^2 \dd f_0 + \int_0^t \left(\int |r|^2 \dd f_s\right)^{1/2} \; \left(\int |v- \bar v|^2 \dd f_s\right)^{1/2} \dd s,  
\end{align*}
and thus 
\begin{align*}
\left(\int |r|^2 \dd f_t\right)^{1/2} 
&\le 1+ \int |r|^2 \dd f_0 + \int_0^t \left(\int |r|^2 \dd f_s\right)^{1/2} \; \left(\int |v- \bar v|^2 \dd f_s\right)^{1/2} \dd s,
\end{align*}
where we used $x > \sqrt{x}$ for $x>1$.

We define $\alpha := (1+\int |r|^2 \dd f_0)$ and $\beta:=\left(\int |v-\bar v|^2 \dd f_0+\frac{\| \nabla \V \|_\infty^2}{\varepsilon} \right)$
Then we obtain with Gronwall inequality
\[
\left(\int |r|^2 \dd f_t\right)^{1/2} \le \alpha \exp( \int_0^t \beta e^{-(\underline{\psi} - \varepsilon)s }\dd s ).
\]
Since the integral converges, also the support of $f_t$ w.r.t.~$r$ is compactly supported for all $t\ge 0$. 

As the support of $f_t$ is compactly supported for all times, we are allowed to apply LaSalle's invariance principle to obtain the result.
\end{proof}

\section{PHS preserving coupling of subsystems}\label{sec:coupling}
To discuss strategies for the coupling of subsystems, we begin with the identification of the ports of the generalized mass, spring and damper components which model the interacting particle system.

\subsection{Identification of ports}
The interaction dynamics can be interpreted as a generalized mass-spring-damper system. In order to identify this we decouple the system into its smallest parts. This allows us then to discuss PHS preserving coupling of different subsystems.

The $i$-th mass is described by its position $x_i$ and its velocity $v_i$. Its evolution is driven by its kinetic energy $H^m(x_i,v_i) = v_i^\top v_i$ and deviations from this velocity are due to external forces $f_i$ which are called flows in the PHS framework. Together with the effort $e_i = \frac{\partial}{\partial v_i} H^m(x_i,v_i) = v_i$ this is leading to the dynamics
\begin{align*}
\frac{\dd}{\dd t} x_i = v_i, \quad \frac{\dd}{\dd t} v_i =  f_i, \quad \Rightarrow \quad \frac{\dd}{\dd t} z_i = J \nabla H^m(z_i) + B_i u_i, \quad y_i = B_i^\top \nabla H^m(z) =  v_i,
\end{align*}
where $z_i=(x_i,v_i)$, $J = \left(\begin{smallmatrix}0 & I \\ -I & 0\end{smallmatrix}\right)$,  $B_i = \left(\begin{smallmatrix}0 \\ I\end{smallmatrix}\right)$  and $u_i =  f_i$ is the input and $y_i=v_i=e_i$ the output.

The spring and damper connecting mass $i$ and mass $j$ are described by the relative position $q_{ij}:=x_i-x_j$ and the relative velocity $v_{ij} := v_i - v_j$. Note that the damper is a purely dissipative element, hence it admits no Hamiltonian but the force $f_{ij}^\mathrm{damper} = -\psi(|q_{ij}|)v_{ij}$. The Hamiltonian of the spring is given by $H^{sd}(q_{ij},v_{ij}) = \calV(q_{ij})$. Altogether this leads to the dynamics
\begin{align*}
&\frac{\dd}{\dd t} q_{ij} = v_{ij}, \qquad \frac{\dd}{\dd t} v_{ij} = -\frac{\partial}{\partial q_{ij}} H^{sd}(q_{ij},v_{ij}) -\psi(|q_{ij}|)v_{ij} \\ 
&\frac{\dd}{\dd t}z_{ij} = (J - R(z_{ij})) \nabla H^{sd}(z_{ij}) + B_{ij}u_{ij},
\end{align*}
where $z_{ij}=(q_{ij},v_{ij})$, $J = \left(\begin{smallmatrix} 0&I\\-I  & 0\end{smallmatrix}\right)$, $R(z_{ij}) = \left(\begin{smallmatrix}0 & 0 \\ 0 & \psi(|q_{ij}|)\end{smallmatrix}\right), B_{ij} = \left(\begin{smallmatrix} I &  -I\\ 0&0\end{smallmatrix}\right)$ and $u_{ij} = \left(\begin{smallmatrix}v_i \\ v_j\end{smallmatrix}\right)$. The corresponding output is given by
\[
y_{ij} = B_{ij}^\top \nabla H^{sd}(z_{ij}) = \begin{pmatrix} I & 0 \\ -I&0\end{pmatrix} \nabla H^{sd}(z_{ij}) = \begin{pmatrix} I & 0 \\ -I&0\end{pmatrix} \begin{pmatrix} \nabla \calV(q_{ij}) \\0\end{pmatrix} = \begin{pmatrix} \;\; \nabla \calV(q_{ij}) \\ -\nabla \calV(q_{ij})\end{pmatrix}.
\]
The flow  and effort variable of the spring-damper system connecting mass $i$ and mass $j$ are given by  
 \[f_{ij} = v_i - v_j \quad \textrm{ and } \quad  e_{ij} = \frac{\partial}{\partial q_{ij}} H^{sd}(q_{ij},v_{ij})+\psi(|q_{ij}|)v_{ij}=\nabla \calV(q_{ij})+\psi(|q_{ij}|)v_{ij}.\]

To couple the $i$-th mass with spring-damper system connecting mass $i$ and mass $j$ we define
$f_i=-e_{ij}$ and $e_i=f_{ij}$ leading to 
\begin{align*}
\frac{\dd}{\dd t} x_i = v_i, \quad \frac{\dd}{\dd t} v_i = -\nabla \calV(q_{ij})-\psi(|q_{ij}|)v_{ij}.
\end{align*}
Taking all binary interactions into account and summation over all binary interactions yields \eqref{eq:xvparticles}, which leads to the PHS formulation in closed  form as given in \eqref{eq:PHSparticlenew}.
Alternatively, instead of considering the actual positions $x_i$ of the particles, we can use the relative positions $q_{ij}$ which then yields the PHS formulation studied in 
\cite{Matei}.

\subsection{Coupling of identical subsystems}
On the mean-field level we can easily couple two interacting particle system of same type by adding and rescaling their probability measures. Indeed, let $f^1, f^2$ be two interacting particle systems with identical interaction behaviour. Then 
\begin{align*}
    \partial_t f^1 = - \nabla_{(x,v)} \cdot \Big( \big(f^1\ast (A \nabla_f \calH(f^1)) \big) f^1 \Big), \qquad \partial_t f^2 = - \nabla_{(x,v)} \cdot \Big( \big(f^2\ast (A \nabla_f \calH(f^2)) \big) f^2 \Big)
\end{align*}
Let us now define $f^\mathrm{sum} = (f^1 + f^2)/(\int df^1 + \int df^2)$, where we rescale to obtain a probability measure. Note that if $f^1$ and $f^2$ are probability measures, it holds
\[
f^\mathrm{sum} = \frac12(f^1 + f^2).
\]
The Hamiltonian of the coupled system is given by
\[
H(f_t^\mathrm{sum}) = \int_{\mathbb R^d\times \mathbb R^d}\int_{\R^d \times \R^d} \left(  \frac{1}{2}  v^T v + \V(x-\bar x)\right) \dd f^\mathrm{sum}(t,\bar x,\bar v) \dd f^\mathrm{sum}(t,x,v)
\]
leading to the dynamics
\begin{align*}
    \partial_t f^\mathrm{sum} = - \nabla_{(x,v)} \cdot \Big( \big(f^\mathrm{sum} \ast (A \nabla_f \calH(f^\mathrm{sum})) \big) f^\mathrm{sum} \Big).
\end{align*}

If we sample now $N$ particles from $f^1(0,x,v)$ and $N$ particles from $f^2(0,x,v)$ we obtain the systems
\begin{align*}
    \frac{d}{dt} x_i^1 &= v_i^1, \\
    \frac{d}{dt} v_i^1 &= -\frac{1}{N} \sum_{j=1}^{N} \psi(| x_i^1 - x_j^1  |)(v_i^1 - v_j^1) - \frac{1}{N} \sum_{j=1}^{N} \nabla \calV(x_i^1 - x_j^1)
\end{align*}
and
\begin{align*}
    \frac{d}{dt} x_i^2 &= v_i^2, \\
    \frac{d}{dt} v_i^2 &= -\frac{1}{N} \sum_{j=1}^{N} \psi(| x_i^2 - x_j^2  |)(v_i^2 - v_j^2) - \frac{1}{N} \sum_{j=1}^{N} \nabla \calV(x_i^2 - x_j^2).
\end{align*}
Sampling $2N$ particles from $f^\mathrm{sum}$ yields
\begin{align*}
    \frac{d}{dt} x_i &= v_i, \\
    \frac{d}{dt} v_i &= -\frac{1}{2N} \sum_{j=1}^{2N} \psi(| x_i - x_j  |)(v_i - v_j) - \frac{1}{2N} \sum_{j=1}^{2N} \nabla \calV(x_i - x_j)
\end{align*}
a fully coupled system. Here, we stack the vectors $x=(x^1,x^2)$ and $v=(v^1,v^2)$. The Hamiltonian as well as the system matrices $J$ and $R$ admit the same structure in the dimension $2N$. Note that the generalization to $K\in\N$ coupled interacting particle systems of same type is straightforward.

\begin{remark}\label{rem:Hpreserving}
We want to stress the fact that in general the value of the Hamiltonian of the coupled systems is greater than the value of the sum of the Hamiltonians of the subsystems. This is due to the fact that additional generalized springs are needed to to define the interaction behaviour of the individuals of the different swarms. However, in case of two identical swarms as above the rescaling of the Hamiltonian yields that the values of the Hamiltonian coincide. In the following we describe one approach that allows the coupling different subsystems in a Hamiltonian preserving way. In fact, the interacting across subsystems influences only the alignment.
\end{remark}

\subsection{Coupling of different species}
Let us consider interacting particle of two different species, which are modelled with the help of different interaction potentials $\calV_k$ and $\psi_k$ for $k=1,2.$ For simplicity we assume that both subsystem consist of $N\in\N$ particles. The dynamics of the subsystems read
\begin{align*}
    \frac{d}{dt} x_i^k &= v_i^k, \\
    \frac{d}{dt} v_i^k &= -\frac{1}{N} \sum_{j=1}^{N} \psi_k(| x_i^k - x_j^k  |)(v_i^k - v_j^k) - \frac{1}{N} \sum_{j=1}^{N} \nabla \calV_k(x_i^k - x_j^k)
\end{align*}
with Hamiltonian
$$ \calH^N_k(z^k(t)) = \sum_{i=1}^{N} (v_i^k(t))^\top v_i^k(t) + \frac{1}{N}\sum_{i,j=1}^{N} \calV_k(x_i^k(t) - x_j^k(t)).  $$
In order to interconnect the systems in a power conserving way, we have to define how particles of different species interact with each other. 
By $\psi_c:\mathbb R\rightarrow \mathbb R_{\ge 0}$ we denote this interaction term. Then we obtain 
\begin{align*}
    \frac{d}{dt} x_i^k &= v_i^k, \\
    \frac{d}{dt} v_i^k &= -\frac{1}{N} \sum_{j=1}^{N} \psi_k(| x_i^k - x_j^k |)(v_i^k - v_j^k) - \frac{1}{N} \sum_{j=1}^{N} \psi_{c}(| x_i^k - x_j^\ell |)(v_i^k - v_j^\ell) \\
    &\qquad - \frac{1}{N} \sum_{j=1}^{N} \nabla \calV_k(x_i^k - x_j^k),
\end{align*}
where $i=1,\ldots, N$, $k,\ell\in\{1,2\}$ and $k+\ell=3$.
The Hamiltonian is then given by
$$\calH^N(z(t)) = \calH_1^N(z^1(t)) + \calH_2^N(z^2(t)). 
$$
Let us define $z=(x^1, v^1, x^2, v^2).$ Then
\begin{equation}\label{eq:PHSparticle2}
\frac{\dd}{\dd t} z = \left[ \begin{pmatrix} 0 & I & 0 & 0  \\ -I & 0 & 0 & 0 \\ 0 & 0 & 0 & I \\ 0 & 0 & -I & 0  \end{pmatrix} - \begin{pmatrix} 0 & 0 & 0 & 0 \\ 0 & \Psi_1(z) & 0 & \Psi_{c}(z) \\  0 & 0 & 0 & 0 \\ 0 & \Psi_{c}^\top(z) & 0 & \Psi_{2}(z)\end{pmatrix} \right] \frac{\partial \calH^N(z)}{\partial z}. 
\end{equation}
with 
 \begin{align*}  \Psi_k(z) =& \frac{1}{N}\!\begin{pmatrix} \displaystyle\sum_{j=2}^N \!\psi_k(|x^k_j-x^k_1|) &-\psi_k(|x^k_2-x^k_1|) & \ldots & -\psi_k(|x^k_N-x^k_1|)\\
 -\psi_k(|x^k_1-x^k_2|) & \displaystyle\sum_{j=1,j\ne 2}^N \!\psi_k(|x^k_j-x^k_2|)& \ddots & \vdots\\
 \vdots &\ddots &\ddots&-\psi_k(|x^k_N\!-\!x^k_{N-1}|)\\
 -\psi_k(|x^k_1-x^k_N|)&\ldots&-\psi_k(|x^k_{N-1}\!-\!x^k_N|)&\displaystyle\sum_{j=1}^{N-1} \!\psi_k(|x^k_j-x^k_N|)\end{pmatrix}\\
 &+\frac{1}{N}\!\begin{pmatrix} \displaystyle\sumj \!\psi_c(|x^\ell_j-x^k_1|) &0 & \ldots & 0\\
 0 & \displaystyle\sumj \!\psi_c(|x^\ell_j-x^k_2|)& \ddots & \vdots\\
 \vdots &\ddots &\ddots&0\\
 0&\ldots&0&\displaystyle\sumj \!\psi_c(|x^\ell_j-x^k_N|)\end{pmatrix},
 \end{align*}
and 
\[  \Psi_c(z) =-\frac{1}{N}\!\begin{pmatrix} \psi_{c}(|x_1^2-x_1^1|) &\psi_{c}(|x_2^2-x_1^1|) & \ldots & \psi_{c}(|x_{N}^2-x_1^1|)\\
\psi_{c}(|x_1^2-x_2^1|) & \psi_{c}(|x_2^2-x_2^1|) & \ddots & \vdots\\
\vdots &\ddots &\ddots&\psi_{c}(|x_{N}^2\!-\!x_{N-1}^1|)\\
\psi_{c}(|x_1^2-x_{N}^1|)&\ldots&\psi_{c}(|x_{N-1}^2\!-\!x_{N}^1|)&\psi_{c}(|x_N^2-x_N^1|) \end{pmatrix}.\]
Now, there are different cases: passing both species to the mean-field limit, passing just one to the limit and the other remains finite.
For the mixed case we obtain
\begin{align*}
\partial_t f &+ \nabla_x \cdot \left( v f \right) \\
&= \nabla_v \cdot \left(\left( \int \psi_1(|x - \bar x|)( v-\bar v ) + \nabla \mathcal V_1(x - \bar x) \dd f(t,\bar x,\bar v) \right. \right. \\
& \qquad \qquad\quad + \left.\left. \frac{1}{N}\sum_{j=1}^N \psi_{c}(| x - x_j |)(v - v_j)  \right) f \right), \\
\frac{\dd}{\dd t} x_i &= v_i, \\
\frac{d}{dt} v_i &=  -\frac{1}{N} \sum_{j=1}^{N} \psi_2(| x_i - x_j |)(v_i - v_j) - \int  \psi_{c}(| x_i - \bar x |)(v_i - \bar v) \dd f(\bar x, \bar v)\\
    &\qquad - \frac{1}{N} \sum_{j=1}^{N} \nabla \calV_2(x_i^k - x_j^k).
\end{align*}
The Hamiltonian is a combination as well
$\calH(f_t,z(t)) = \calH(f_t) + \calH^N(z(t))$ as the coupling across subsystems affects only the alignment terms, see Remark~\ref{rem:Hpreserving} for more details.

\section{Conclusion and outlook}
We derived a minimal port-Hamiltonian formulation of interacting particle systems and  showed that the port-Hamiltonian structure is preserved in the mean-field limit. The Hamiltonian is used as Lyapunov function to characterize the long-time behavior of the systems on the particle and the mean-field level. Hence the PHS formulation opens a new perspective on the well-studied particle and mean-field description of interacting particle system. Moreover, the identification of Casimir functions prepares the ground to define port-Hamiltonian  preserving control strategies in future work. On the other hand, the LaSalle-type argument for the long-term behavior may open the door for convergence results for general consensus dynamics for optimization and sampling tasks in the spirit of \cite{carrillo2022consensus,totzeck2022trends}.

\section*{Acknowledgments} We thank the anonymous referee for their critical and constructive comments that helped us to strengthen the results on the long-time behavior in both, the ODE and the PDE setting.
        
        \putbib[interparticle]
    \end{bibunit}

\end{document}